\documentclass[reqno]{amsart}
\usepackage{amssymb}
\usepackage{graphicx}

\usepackage[usenames, dvipsnames]{color}
\usepackage{verbatim}
\usepackage{mathrsfs}
\usepackage{bm}
\usepackage{cite}

\numberwithin{equation}{section}

\newtheorem{theorem}{Theorem}[section]
\newtheorem{corollary}[theorem]{Corollary}
\newtheorem{lemma}[theorem]{Lemma}
\newtheorem{prop}[theorem]{Proposition}

\theoremstyle{definition}
\newtheorem{remark}[theorem]{Remark}

\theoremstyle{definition}
\newtheorem{definition}[theorem]{Definition}

\theoremstyle{definition}

\makeatletter
\def\dashint{\operatorname%
{\,\,\text{\bf-}\kern-.98em\DOTSI\intop\ilimits@\!\!}}
\makeatother

\def\\det{\text{\det}}

\def\Xint#1{\mathchoice
 {\XXint\displaystyle\textstyle{#1}}%
 {\XXint\textstyle\scriptstyle{#1}}%
 {\XXint\scriptstyle\scriptscriptstyle{#1}}%
 {\XXint\scriptscriptstyle\scriptscriptstyle{#1}}%
 \!\int}
\def\XXint#1#2#3{{\setbox0=\hbox{$#1{#2#3}{\int}$}
  \vcenter{\hbox{$#2#3$}}\kern-.5\wd0}}

\def\dashint{\Xint-}

\def\.5{\frac{1}{2}}

\newcommand{\RN}[1]{%
  \textup{\uppercase\expandafter{\romannumeral#1}}%
}

\renewcommand{\epsilon}{\varepsilon}

\newcounter{marnote}

\begin{document}

\title[Local regularity for anisotropic weighted elliptic and parabolic equations]{Local regularity for nonlinear elliptic and parabolic equations with anisotropic weights}

\author[C.X. Miao]{Changxing Miao}
\address[C.X. Miao] {1. Beijing Computational Science Research Center, Beijing 100193, China.}
\address{2. Institute of Applied Physics and Computational Mathematics, P.O. Box 8009, Beijing, 100088, China.}
\email{miao\_changxing@iapcm.ac.cn}

\author[Z.W. Zhao]{Zhiwen Zhao}

\address[Z.W. Zhao]{Beijing Computational Science Research Center, Beijing 100193, China.}
\email{zwzhao365@163.com}


\date{\today} 


\maketitle
\begin{abstract}
The main purpose of this paper is to capture the asymptotic behavior for solutions to a class of nonlinear elliptic and parabolic equations with the anisotropic weights consisting of two power-type weights of different dimensions near the degenerate or singular point, especially covering the weighted $p$-Laplace equations and weighted fast diffusion equations. As a consequence, we also establish the local H\"{o}lder estimates for their solutions in the presence of single power-type weights.
\end{abstract}

\section{Introduction and main results}

In this paper, we aim to study the local behavior of solutions to a class of nonlinear elliptic and parabolic equations with the weights comprising two power-type weights of different dimensions. The novelties of this paper are two-fold: one is that the weights embodied in nonlinear elliptic and parabolic equations are anisotropic; another is the establishment of anisotropic weighted Poincar\'{e} inequality in Section \ref{SEC002} below, which is a crucial tool to application of the De Giorgi truncation method \cite{D1957} in the following proofs. In particular, the anisotropy of the weights will bring great difficulties of analyses and computations for both establishing anisotropic weighted Poincar\'{e} inequality and applying the De Giorgi truncation method. The mathematical formulations and main results for the considered nonlinear elliptic and parabolic problems with anisotropic weights will be, respectively, presented as follows.

\subsection{The nonlinear elliptic equations with anisotropic weights}
Consider a bounded smooth domain $\Omega\subset\mathbb{R}^{n}$ with $0\in\Omega$ and $n\geq2$. With regard to the weighted elliptic equations, we mainly study the local regularity of solution to the following problem
\begin{align}\label{PROBLEM001ZW}
\begin{cases}
\mathrm{div}(Aw|\nabla u|^{p-2}\nabla u)=0,& \mathrm{in}\;\Omega,\\
0\leq u\leq\overline{M}<\infty,&\mathrm{in}\;\Omega,
\end{cases}
\end{align}
where $w=|x'|^{\theta_{1}}|x|^{\theta_{2}}$, the values of $\theta_{1}$ and $\theta_{2}$ are assumed in the following theorems, $1<p<n+\theta_{1}+\theta_{2}$, $\overline{M}$ is a given positive constant, $A(x)=(a_{ij}(x))_{n\times n}$ is symmetric and satisfies
\begin{align}\label{WQ090}
\lambda^{-1}|\xi|^{2}\leq\sum^{n}_{i,j=1}a_{ij}(x)\xi_{i}\xi_{j}\leq\lambda|\xi|^{2},\quad\lambda\geq1,\; \text{for a.e. $x\in\Omega$ and all $\xi\in\mathbb{R}^{n}$.}
\end{align}
Here and throughout this paper, we use superscript prime to denote $(n-1)$-dimensional variables and domains, such as $x'$ and $B'$. Moreover, in the following we simplify the notations $B_{R}(0)$ and $B_{R}'(0')$ as $B_{R}$ and $B'_{R}$, respectively, where $R>0$. The prototype equation is the anisotropic weighted $p$-Laplacian, that is, the equation in the case when $A=I$ in \eqref{PROBLEM001ZW}. Remark that the origin can be called the degenerate or singular point of the weight. For example, if $\theta_{1}>0$, $\theta_{2}>0$, then the weight $w=|x'|^{\theta_{1}}|x|^{\theta_{2}}\rightarrow0,$ as $|x|\rightarrow0$, while, for $\theta_{1}<0$, $\theta_{2}<0$, it blows up as $|x|$ tends to zero. For the former, the origin is called the degenerate point of the weight, while it is called the singular point for the latter.

For the weighted elliptic problem \eqref{PROBLEM001ZW}, Fabes, Kenig and Serapioni \cite{FKS1982} established the local H\"{o}lder regularity of weak solutions under the case of $\theta_{1}=0,\,\theta_{2}>-n$ and $p=2$. However, the value of H\"{o}lder index $\alpha$ obtained in \cite{FKS1982} is not explicit. Recently, Dong, Li and Yang \cite{DLY2022} utilized spherical harmonic expansion to find the exact value of index $\alpha$ for the solution near the degenerate point of the weight. To be precise, for problem \eqref{PROBLEM001ZW} with $\Omega$ replaced by $B_{R}$, $R>0$, let $n\geq2$, $\theta_{1}=0$, $\theta_{2}=p=2$ and $A=\kappa(x)I$, where $\kappa$ satisfies that $\lambda^{-1}\leq\kappa\leq\lambda$ in $B_{R}$ and $\int_{\mathbb{S}^{n-1}}\kappa x_{i}=0$, $i=1,2,...,n$. Based on these assumed conditions, they derived
\begin{align*}
u(x)=u(0)+O(1)|x|^{\alpha},\quad\alpha=\frac{-n+\sqrt{n^{2}+4\tilde{\lambda}_{1}}}{2},\quad\mathrm{in}\;B_{R/2},
\end{align*}
where $O(1)$ represents some quantity such that $|O(1)|\leq C=C(n,\lambda,\overline{M}),$ $\tilde{\lambda}_{1}\leq n-1$ is the first nonzero eigenvalue of the following eigenvalue problem:
\begin{align*}
-\mathrm{div}_{\mathbb{S}^{n-1}}(\kappa(\xi)\nabla_{\mathbb{S}^{n-1}}u(\xi))=\tilde{\lambda} \kappa(\xi) u(\xi),\quad\xi\in\mathbb{S}^{n-1}.
\end{align*}
In particular, $\tilde{\lambda}_{1}=n-1$ if $A=I$. See Lemmas 2.2 and 5.1 in \cite{DLY2022} for more details. By finding the explicit exponent $\alpha$, they succeeded in solving the optimal gradient blow-up rate for solution to the insulated conductivity problem in dimensions greater than two, which has been previously regarded as a challenging problem. By their investigations in \cite{DLY2021,DLY2022}, we realize that the H\"{o}lder regularity for solutions to the weighted elliptic problem \eqref{PROBLEM001ZW} is in close touch with the insulated conductivity problem arising from composite materials. Then the study on the regularity for weighted elliptic problem \eqref{PROBLEM001ZW} is a topic of theoretical interest and also of great relevance to applications for the insulated composites. It is worth emphasizing that when $p>2$, the exact value of index $\alpha$ still remains open. In addition, with regard to the H\"{o}lder regularity for nonlinear degenerate elliptic equations without weights, we refer to \cite{L2019,S1964,T1967} and the references therein.

Before stating the definition of weak solution to problem \eqref{PROBLEM001ZW}, we first introduce some notations. Throughout this paper, we will use $L^{p}(\Omega,w)$ and $W^{1,p}(\Omega,w)$ to represent weighted $L^{p}$ space and weighted Sobolev space with their norms, respectively, written as
\begin{align*}
\begin{cases}
\|u\|_{L^{p}(\Omega,w)}=\left(\int_{\Omega}|u|^{p}wdx\right)^{\frac{1}{p}},\\
\|u\|_{W^{1,p}(\Omega,w)}=\left(\int_{\Omega}|u|^{p}wdx\right)^{\frac{1}{p}}+\left(\int_{\Omega}|\nabla u|^{p}wdx\right)^{\frac{1}{p}}.
\end{cases}
\end{align*}
We say that $u\in W^{1,p}(\Omega,w)$ is a weak solution of problem \eqref{PROBLEM001ZW} if
\begin{align*}
\int_{\Omega}Aw|\nabla u|^{p-2}\nabla u\cdot\nabla\varphi dx=0,\quad\forall\,\varphi\in W_{0}^{1,p}(\Omega,w).
\end{align*}

For later use, we introduce the following indexing sets:
\begin{align*}
\begin{cases}
\mathcal{A}=\{(a,b): a>-(n-1),\,b\geq0\},\\
\mathcal{B}=\{(a,b):a>-(n-1),\,b<0,\,a+b>-n\},\\
\mathcal{C}_{q}=\{(a,b):a<(n-1)(q-1),\,b\leq0\},\quad q>1,\\
\mathcal{D}_{q}=\{(a,b):a<(n-1)(q-1),\,b>0,\,a+b<n(q-1)\},\quad q>1,\\
\mathcal{F}=\{(a,b):a+b>-(n-1)\}.
\end{cases}
\end{align*}
The local behavior of solution to problem \eqref{PROBLEM001ZW} near the degenerate or singular point of the anisotropic weight is captured as follows.
\begin{theorem}\label{THM001Z}
For $n\geq2$, $(\theta_{1},\theta_{2})\in(\mathcal{A}\cup\mathcal{B})\cap(\mathcal{C}_{q}\cup\mathcal{D}_{q})\cap\mathcal{F}$, $1<q<p<n+\theta_{1}+\theta_{2}$, let $u$ be a weak solution of problem \eqref{PROBLEM001ZW} with $\Omega=B_{1}$. Then there exists a constant $0<\alpha<1$ depending only on $n,p,q,\theta_{1},\theta_{2},\lambda,$ such that
\begin{align}\label{MAIN0011Z}
u(x)=u(0)+O(1)|x|^{\alpha},\quad\text{for all}\;x\in B_{1/2},
\end{align}
where $O(1)$ denotes some quantity satisfying that $|O(1)|\leq C=C(n,p,q,\theta_{1},\theta_{2},\lambda,\overline{M}).$
\end{theorem}
\begin{remark}\label{RE03}
If the considered domain $B_{1}$ is replaced with $B_{R_{0}}$ for any given $R_{0}>0$ in Theorems \ref{THM001Z} and \ref{THM002}, then by applying their proofs with minor modification, we obtain that \eqref{MAIN0011Z} and \eqref{Z010} also hold with $B_{1/2}$ replaced by $B_{R_{0}/2}$. In this case the constant $C$ will depend on $R_{0}$, but the index $\alpha$ not.

\end{remark}

\begin{remark}\label{WQAN016}
The result in Theorem \ref{THM001Z} can be extended to general degenerate elliptic equations as follows:
\begin{align*}
\begin{cases}
\mathrm{div}(\mathcal{G}(x,\nabla u))=0,& \mathrm{in}\;B_{R_{0}},\\
0\leq u\leq\overline{M}<\infty,&\mathrm{in}\;B_{R_{0}},
\end{cases}
\end{align*}
where $R_{0}>0$, $\mathcal{G}:B_{R_{0}}\times\mathbb{R}^{n}\rightarrow\mathbb{R}^{n}$ is a Carath\'{e}odory function such that for a.e. $x\in B_{R_{0}}$ and any $\xi\in\mathbb{R}^{n}$, there holds
\begin{align*}
\lambda^{-1}w(x)|\xi|^{p}\leq\mathcal{G}(x,\xi)\cdot\xi,\quad|\mathcal{G}(x,\xi)|\leq\lambda w(x)|\xi|^{p-1},\quad w(x)=|x'|^{\theta_{1}}|x|^{\theta_{2}}.
\end{align*}
Here $\lambda\geq1$, $(\theta_{1},\theta_{2})\in(\mathcal{A}\cup\mathcal{B})\cap(\mathcal{C}_{q}\cup\mathcal{D}_{q})\cap\mathcal{F}$, $1<q<p<n+\theta_{1}+\theta_{2}$. In fact, it only needs to slightly modify the proof of Lemma \ref{lem003ZZW} below for the purpose of achieving this generalization.

\end{remark}

When $\theta_{1}=0$, the above weight becomes a single power-type weight. In this case, we establish the H\"{o}lder estimates as follows.
\begin{theorem}\label{THM002}
For $n\geq2$, $\theta_{1}=0,\,\theta_{2}>-(n-1)$, $1<p<n+\theta_{2}$, let $u$ be a bounded weak solution of problem \eqref{PROBLEM001ZW} with $\Omega=B_{1}$. Then there exist a small constant $0<\alpha=\alpha(n,p,\theta_{2},\lambda)<1$ and a large constant $0<C=C(n,p,\theta_{2},\lambda,\overline{M})$ such that
\begin{align}\label{Z010}
|u(x)-u(y)|\leq C|x-y|^{\alpha},\quad\text{for all}\;x,y\in B_{1/2}.
\end{align}

\end{theorem}

Observe that when $\theta_{1}=0$, the equation in \eqref{PROBLEM001ZW} will become degenerate elliptic equation in any domain away from the origin, then we can directly establish its H\"{o}lder regularity in these regions by using the interior H\"{o}lder estimates for degenerate elliptic equation. This, in combination with Remark \ref{RE03} and Theorem \ref{THM002}, gives the following corollary.
\begin{corollary}
For $n\geq2$, $\theta_{1}=0,\,\theta_{2}>-(n-1)$, $1<p<n+\theta_{2}$, let $u$ be a weak solution of problem \eqref{PROBLEM001ZW}. Then $u$ is locally H\"{o}lder continuous in $\Omega$, that is, for any compact subset $K\subset\Omega$, there exists two constants $0<\alpha=\alpha(n,p,\theta_{2},\lambda)<1$ and $C=C(\mathrm{dist}(K,\partial\Omega),\mathrm{dist}(0,\partial\Omega),n,p,\theta_{2},\lambda,\overline{M})>0$ such that \eqref{Z010} holds with $B_{1/2}$ replaced by $K$.
\end{corollary}

\subsection{The nonlinear parabolic equations with anisotropic weights}
Let $0\in\Omega\subset\mathbb{R}^{n}$, $n\geq2$ be defined as above. The second problem of interest is concerned with studying the local regularity of solution to the weighted nonlinear parabolic equation as follows:
\begin{align}\label{PROBLEM001}
\begin{cases}
w_{1}\partial_{t}u^{p}-\mathrm{div}(Aw_{2}\nabla u)=0,& \mathrm{in}\;\Omega_{T},\\
0<\overline{m}\leq u\leq\overline{M}<\infty,&\mathrm{in}\;\Omega_{T},
\end{cases}
\end{align}
where $\Omega_{T}=\Omega\times(-T,0]$, $T>0$, $w_{1}=|x'|^{\theta_{1}}|x|^{\theta_{2}}$, $w_{2}=|x'|^{\theta_{3}}|x|^{\theta_{4}}$, $p\geq1$, the ranges of $\theta_{i}$, $i=1,2,3,4$ are prescribed in the following theorems, $\overline{m}$ and $\overline{M}$ are two given positive constants, the symmetric matrix $A=(a_{ij}(x))_{n\times n}$ satisfies the uniformly elliptic condition in \eqref{WQ090}. When $\theta_{i}=0$, $i=1,2,3,4$ and $A=I$, equation in \eqref{PROBLEM001} becomes fast diffusion equation, whose relevant mathematical problem is modeled by
\begin{align}\label{PROBLEM009}
\begin{cases}
\partial_{t}u^{p}-\Delta u=0,&\mathrm{in}\;\Omega\times(0,\infty),\\
u=0,&\mathrm{on}\;\partial\Omega\times(0,\infty),\\
u(x,0)=u_{0}(x)\geq0.&
\end{cases}
\end{align}
In physics, equation in \eqref{PROBLEM009} can be used to describe fast diffusion phenomena occurring in gas-kinetics, plasmas and thin liquid film
dynamics. For more related applications and physical explanations, see \cite{DK2007,V2007} and the references therein.

For problem \eqref{PROBLEM009}, it is well known that when $u_{0}(x)\not\equiv0$, there exists a finite extinction time $T^{\ast}>0$ such that $u(\cdot,t)\equiv0$ in $\Omega$ if $t\in[T^{\ast},\infty)$ and $u(\cdot,t)>0$ in $\Omega$ if $t\in(0,T^{\ast})$. This, together with the continuity of $u$ (see Chen-DiBenedetto \cite{CD1988}), indicates that for any $U\subset\subset\Omega\times(0,T^{\ast})$, there exist two positive constants $\overline{m}$ and $\overline{M}$ such that $0<\overline{m}\leq u\leq\overline{M}<\infty$ for $(x,t)\in U$. This fact motivates our investigation on the local regularity of weak solution for the corresponding weighted problem \eqref{PROBLEM001}. In particular, it can be called the weighted fast diffusion equation when $A=I$ in \eqref{PROBLEM001}. For the fast diffusion problem \eqref{PROBLEM009}, the regularity of solution and its asymptotic behavior near extinction time have been extensively studied, for example, see \cite{JX2019,JX2022,CD1988,S1983,K1988,BV2010,DK2007,DK1992,DGV2012,DKV1991} for the regularity and \cite{BH1980,BGV2012,FS2000,BF2021,A2021} for the extinction behavior, respectively. In particular, Jin and Xiong recently established a priori H\"{o}lder estimates for the solution to a weighted nonlinear parabolic equation in Theorem 3.1 of \cite{JX2022}, which is critical to the establishment of optimal global regularity for fast diffusion equation with any $1<p<\infty$. Their results especially answer the regularity problem proposed by Berryman and Holland \cite{BH1980}. It is worth pointing out that the degeneracy of weight in \cite{JX2022} is located at the boundary. By contrast, the degeneracy or singularity of the weights considered in this paper lies in the interior.  This will lead to some distinct differences in terms of the establishments of H\"{o}lder estimates under these two cases. Moreover, since the weights considered in this paper take more sophisticated forms comprising two power-type weights of different dimensions, it greatly increases the difficulties of analyses and calculations. With regard to the regularity for weighted parabolic problem in the case when $p=1$ in \eqref{PROBLEM001}, we refer to \cite{CS1984,GW1991,S2010} and the references therein.

The weighted $L^{p}$ space and weighted Sobolev spaces with respect to space variable have been defined above. Similarly, for a weight $w$, let $W^{1,p}(\Omega_{T},w)$ represent the corresponding weighted Sobolev spaces in $(x,t)$ with its norm given by
\begin{align*}
\|u\|_{W^{1,p}(\Omega_{T},w)}=\left(\int_{\Omega_{T}}|u|^{p}wdxdt\right)^{\frac{1}{p}}+\left(\int_{\Omega_{T}}(|\partial_{t}u|^{p}+|\nabla u|^{p})dxdt\right)^{\frac{1}{p}}.
\end{align*}
We say that $u\in W^{1,2}(\Omega_{T},w_{2})$ is a weak solution of problem \eqref{PROBLEM001} if
\begin{align*}
\int^{t_{2}}_{t_{1}}\int_{\Omega}(w_{1}\partial_{t}u^{p}\varphi+Aw_{2}\nabla u\nabla\varphi)dxdt=0,
\end{align*}
for any $-T\leq t_{1}<t_{2}\leq0$ and $\varphi\in C^{1}(\Omega_{T})$ which vanishes on $\partial\Omega\times(-T,0)$.

Introduce the following index conditions:
\begin{itemize}
{\it
\item[(\bf{S1})] let $n\geq4$ and $1+2/(n-1)<q<2$, if $(\theta_{1},\theta_{2})\in(\mathcal{A}\cup\mathcal{B})\cap\mathcal{C}_{q}$;
\item[(\bf{S2})] let $n\geq3$ and $1+2/n<q<2$, if $(\theta_{1},\theta_{2})\in\mathcal{A}\cap\mathcal{D}_{q}$.}
\end{itemize}
For the local behavior of solution to problem \eqref{PROBLEM001}, we have
\begin{theorem}\label{ZWTHM90}
Suppose that $p\geq1$, $(\theta_{1},\theta_{2})$ satisfies condition $\mathrm{(}\mathbf{S}1\mathrm{)}$ or $\mathrm{(}\mathbf{S}2\mathrm{)}$, $(\theta_{3},\theta_{4})\in\mathcal{A}\cup\mathcal{B},$ $\theta_{1}+\theta_{2}\geq\theta_{3}+\theta_{4}=2$, $\theta_{1}/\theta_{3}=\theta_{2}/\theta_{4}$, $\theta_{3},\theta_{4}\neq0$. Let $u$ be a weak solution of problem \eqref{PROBLEM001} with $\Omega\times(-T,0]=B_{1}\times(-1,0]$. Then there exists a small constant $0<\alpha=\alpha(n,p,q,\theta_{1},\theta_{2},\theta_{3},\lambda,\overline{m},\overline{M})<1$ such that for any $t_{0}\in (-1/4,0)$,
\begin{align}\label{QNAW001ZW}
u(x,t)=u(0,t_{0})+O(1)\big(|x|+\sqrt[\theta_{1}+\theta_{2}]{|t-t_{0}|}\big)^{\alpha},\;\,\forall\, (x,t)\in B_{1/2}\times(-1/4,t_{0}],
\end{align}
where $O(1)$ satisfies that $|O(1)|\leq C=C(n,p,q,\theta_{1},\theta_{2},\theta_{3},\lambda,\overline{m},\overline{M}).$
\end{theorem}
\begin{remark}
We here give explanations for index conditions $\mathrm{(}\mathbf{S}1\mathrm{)}$ and $\mathrm{(}\mathbf{S}2\mathrm{)}$. Observe that if $(\theta_{1},\theta_{2})\in C_{q}$, $1<q<2$ and $\theta_{1}+\theta_{2}\geq2$, then we have $(n-1)(q-1)>2$, which requires that $n\geq4$ and $q>2/(n-1)+1$. Similarly, if $(\theta_{1},\theta_{2})\in D_{q}$, $1<q<2$ and $\theta_{1}+\theta_{2}\geq2$, it requires that $n\geq3$ and $q>2/n+1$.
\end{remark}
\begin{remark}\label{REM001}
For any fixed $R_{0}>0$, let $B_{R_{0}}\times(-R_{0}^{\theta_{1}+\theta_{2}},0]$ substitute for $B_{1}\times(-1,0]$ in Theorems \ref{ZWTHM90} and \ref{THM060}. Then applying their proofs with a slight modification, we derive that \eqref{QNAW001ZW}--\eqref{MWQ099} still hold with $t_{0}\in(-1/4,0)$, $B_{1/2}\times(-1/4,t_{0}]$ and $B_{1/2}\times(-1/4,0)$ replaced by $t_{0}\in(-R^{\theta_{1}+\theta_{2}}_{0}/4,0)$, $B_{R_{0}/2}\times(-R^{\theta_{1}+\theta_{2}}_{0}/4,t_{0}]$ and $B_{R_{0}/2}\times(-R^{\theta_{2}}_{0}/4,0)$, respectively. A difference lies in that the constant $C$ will depend on $R_{0}$, but $\alpha$ not.

\end{remark}

In the case of $\theta_{1}=\theta_{3}=0$, $w_{1}$ and $w_{2}$ become single power-type weight. Then we have
\begin{theorem}\label{THM060}
For $p\geq1$, $n\geq2$, $\theta_{1}=\theta_{3}=0$, $\theta_{2}\geq\theta_{4}=2$, let $u$ be a weak solution of problem \eqref{PROBLEM001} with $\Omega\times(-T,0]=B_{1}\times(-1,0]$. Then there exist two constants $0<\alpha<1$ and $C>0$, both depending only on $n,p,\theta_{2},\lambda,\overline{m},\overline{M},$ such that
\begin{align}\label{MWQ099}
|u(x,t)-u(y,s)|\leq C\big(|x-y|+\sqrt[\theta_{2}]{|t-s|}\big)^{\alpha},
\end{align}
for any $(x,t),(y,s)\in B_{1/2}\times(-1/4,0).$
\end{theorem}

When $\theta_{1}=\theta_{3}=0$ and $\theta_{2}\geq\theta_{4}=2$, the equation in \eqref{PROBLEM001} will be uniformly parabolic in any domain away from the origin. Then we can directly use the interior H\"{o}lder estimates for uniformly parabolic equation to obtain its H\"{o}lder regularity in these regions. This, together with Remark \ref{REM001} and Theorem \ref{THM060}, leads to the following corollary.

\begin{corollary}
For $p\geq1$, $n\geq2$, $\theta_{1}=\theta_{3}=0$, $\theta_{2}\geq\theta_{4}=2$, let $u$ be a weak solution of problem \eqref{PROBLEM001}. Then $u$ is locally H\"{o}lder continuous in $\Omega\times(-T,0)$, that is, for any compact subset $K\subset\Omega\times(-T,0)$, there exist a small constant $0<\alpha=\alpha(n,p,\theta_{2},\lambda,\overline{m},\overline{M})<1$ and a large constant $C=C(\mathrm{dist}(K,\partial(\Omega\times(-T,0))),\mathrm{dist}(0,\partial\Omega),n,p,\theta_{2},\lambda,\overline{m},\overline{M})>0$ such that \eqref{MWQ099} holds with $K$ substituting for $B_{1/2}\times(-1/4,0)$.
\end{corollary}

The rest of this paper is organized as follows. In Section \ref{SEC002} we establish the anisotropic weighted Poincar\'{e} type inequality and its corresponding isoperimetric inequality. Then we make use of the De Giorgi truncation method \cite{D1957} to study the local regularity for solutions to the nonlinear elliptic and parabolic equations with anisotropic weights in Sections \ref{SEC003} and \ref{SEC00four}, respectively.

\section{Anisotropic weighted Poincar\'{e} inequality and its application to the isoperimetric inequality}\label{SEC002}

Before applying the De Giorgi truncation method in \cite{D1957}, we should first establish the corresponding weighted Sobolev and Poincar\'{e} type inequalities. Especially the latter is critical to the establishment of the isoperimetric inequality of De Giorgi type. Recently, Li and Yan \cite{LY2021} solved the anisotropic Caffarelli-Kohn-Nirenberg type inequalities. For the classical Caffarelli-Kohn-Nirenberg type inequalities, see \cite{CKN1984}. We also refer to \cite{L1986,CR2013,CRS2016,BT2002,BCG2006,NS2018,NS2019} for more relevant investigations on weighted Sobolev inequalities.

So it remains to establish the anisotropic weighted Poincar\'{e} type inequality. It will be achieved by proving that $w=|x'|^{\theta_{1}}|x|^{\theta_{2}}$ is an $A_{q}$-weight under the condition of $(\theta_{1},\theta_{2})\in(\mathcal{A}\cup\mathcal{B})\cap(\mathcal{C}_{q}\cup\mathcal{D}_{q})$, see Theorem \ref{THM000Z} and Corollary \ref{QWZM090} below. We here emphasize that the establishment of anisotropic Poincar\'{e} type inequality is one of our main contributions and novelties in this paper.

Denote by $\omega_{n}$ the volume of unit ball in $\mathbb{R}^{n}$. In this section, we employ the notation $a\sim b$ to represent that there exists a constant $C=C(n,\theta_{1},\theta_{2})>0$ such that $\frac{1}{C}b\leq a\leq Cb.$ To begin with, we have
\begin{lemma}\label{MWQAZ090}
$d\mu:=wdx=|x'|^{\theta_{1}}|x|^{\theta_{2}}dx$ is a Radon measure if $(\theta_{1},\theta_{2})\in\mathcal{A}\cup\mathcal{B}$. Moreover, $\mu(B_{R})\sim R^{n+\theta_{1}+\theta_{2}}$ for any $R>0$.

\end{lemma}

%

\begin{proof}
It suffices to verify that the weight $w$ is a locally integrable function in $\mathbb{R}^{n}$ under these above cases. Specifically, it only needs to prove that for any $r>0$, there holds $\mu(B_{r})<\infty$. Recall the following elemental inequalities: for $a,b\geq0$,
\begin{align}\label{INE001}
\begin{cases}
\frac{a^{q}+b^{q}}{2}\leq(a+b)^{q}\leq a^{q}+b^{q},&0\leq q\leq1,\\
a^{q}+b^{q}\leq(a+b)^{q}\leq2^{q-1}(a^{q}+b^{q}),&q>1.
\end{cases}
\end{align}

\noindent{\bf Step 1.} Consider the case when $\theta_{2}\geq0$. Then we have from \eqref{INE001} that
\begin{align*}
\mu(B_{r})=&2\int_{B_{r}\cap\{x_{n}>0\}}|x'|^{\theta_{1}}|x|^{\theta_{2}}dx\sim\int_{B_{r}'}|x'|^{\theta_{1}}dx'\int_{0}^{\sqrt{r^{2}-|x'|^{2}}}(|x'|^{\theta_{2}}+x_{n}^{\theta_{2}})dx_{n}\notag\\
\sim&\int^{r}_{0}\left(s^{n+\theta_{1}+\theta_{2}-2}(r^{2}-s^{2})^{\frac{1}{2}}+s^{n+\theta_{1}-2}(r^{2}-s^{2})^{\frac{\theta_{2}+1}{2}}\right)ds\notag\\
\sim&\,r^{n+\theta_{1}+\theta_{2}}\int^{1}_{0}\left(s^{n+\theta_{1}+\theta_{2}-2}(1-s)^{\frac{1}{2}}+s^{n+\theta_{1}-2}(1-s)^{\frac{\theta_{2}+1}{2}}\right)ds.
\end{align*}
Observe that this type of integration is called Beta function. It makes sense if and only if $n+\theta_{1}+\theta_{2}-1>0$ and $n+\theta_{1}-1>0$. Then the conclusion is proved in the case when $(\theta_{1},\theta_{2})\in\mathcal{A}$.

\noindent{\bf Step 2.} Consider $\theta_{2}<0$. Then it follows from \eqref{INE001} that
\begin{align}\label{WDZM001}
\mu(B_{r})=&2\int_{B_{r}\cap\{x_{n}>0\}}|x'|^{\theta_{1}}|x|^{\theta_{2}}dx\notag\\
\sim&\int_{B_{r}'}|x'|^{\theta_{1}}dx'\int^{\sqrt{r^{2}-|x'|^{2}}}_{0}\frac{1}{|x'|^{-\theta_{2}}+x_{n}^{-\theta_{2}}}dx_{n}.
\end{align}
For the last integration term in \eqref{WDZM001}, we further split it as follows:
\begin{align*}
I_{1}=&\int_{B'_{\frac{1}{\sqrt{2}}r}}|x'|^{\theta_{1}}dx'\int^{|x'|}_{0}\frac{1}{|x'|^{-\theta_{2}}+x_{n}^{-\theta_{2}}}dx_{n},\notag\\
I_{2}=&\int_{B'_{\frac{1}{\sqrt{2}}r}}|x'|^{\theta_{1}}dx'\int^{\sqrt{r^{2}-|x'|^{2}}}_{|x'|}\frac{1}{|x'|^{-\theta_{2}}+x_{n}^{-\theta_{2}}}dx_{n},\notag\\
I_{3}=&\int_{B_{r}'\setminus B'_{\frac{1}{\sqrt{2}}r}}|x'|^{\theta_{1}}dx'\int^{\sqrt{r^{2}-|x'|^{2}}}_{0}\frac{1}{|x'|^{-\theta_{2}}+x_{n}^{-\theta_{2}}}dx_{n}.
\end{align*}
Observe that $|x'|^{-\theta_{2}}\leq|x'|^{-\theta_{2}}+x_{n}^{-\theta_{2}}\leq2|x'|^{-\theta_{2}}$ if $0\leq x_{n}\leq |x'|$. Then for the first term $I_{1}$, we have
\begin{align*}
I_{1}\sim&\int_{B'_{\frac{1}{\sqrt{2}}r}}|x'|^{\theta_{1}}dx'\int^{|x'|}_{0}|x'|^{\theta_{2}}dx_{n}\sim\int^{\frac{r}{\sqrt{2}}}_{0}s^{n+\theta_{1}+\theta_{2}-1}ds\notag\\
\sim&r^{n+\theta_{1}+\theta_{2}}\int^{1}_{0}s^{n+\theta_{1}+\theta_{2}-1}ds.
\end{align*}
This integration makes sense iff $n+\theta_{1}+\theta_{2}>0$.

With regard to the second term $I_{2}$, we divide into two cases to discuss as follow.

{\bf Case 1.} If $\theta_{2}=-1$, then
\begin{align*}
I_{2}\sim&\int_{B'_{\frac{1}{\sqrt{2}}r}}|x'|^{\theta_{1}}dx'\int^{\sqrt{r^{2}-|x'|^{2}}}_{|x'|}x_{n}^{-1}dx_{n}\sim\int^{\frac{r}{\sqrt{2}}}_{0}s^{n+\theta_{1}-2}\ln\frac{\sqrt{r^{2}-s^{2}}}{s}ds\notag\\
\sim&\,r^{n+\theta_{1}-1}\int^{\frac{1}{\sqrt{2}}}_{0}s^{n+\theta_{1}-2}\ln\frac{\sqrt{1-s^{2}}}{s}ds\sim-r^{n+\theta_{1}-1}\int^{\frac{1}{\sqrt{2}}}_{0}s^{n+\theta_{1}-2}\ln s\,ds\notag\\
\sim&-r^{n+\theta_{1}+\theta_{2}}\bigg(s^{n+\theta_{1}+\theta_{2}}\ln s\Big|^{\frac{1}{\sqrt{2}}}_{0}-\int^{\frac{1}{\sqrt{2}}}_{0}s^{n+\theta_{1}+\theta_{2}-1}ds\bigg).
\end{align*}
It makes sense iff $n+\theta_{1}+\theta_{2}>0$.

{\bf Case 2.} If $\theta_{2}\neq-1$, then
\begin{align*}
I_{2}\sim&\int_{B'_{\frac{1}{\sqrt{2}}r}}|x'|^{\theta_{1}}dx'\int^{\sqrt{r^{2}-|x'|^{2}}}_{|x'|}x_{n}^{\theta_{2}}dx_{n}\notag\\
\sim&\int^{\frac{1}{\sqrt{2}}r}_{0}s^{n+\theta_{1}-2}\left((r^{2}-s^{2})^{\frac{\theta_{2}+1}{2}}-s^{\theta_{2}+1}\right)ds\notag\\
\sim&\,r^{n+\theta_{1}+\theta_{2}}\int^{\frac{1}{\sqrt{2}}}_{0}s^{n+\theta_{1}-2}\left((1-s^{2})^{\frac{\theta_{2}+1}{2}}-s^{\theta_{2}+1}\right)ds.
\end{align*}
Note that $\min\{1,2^{-\frac{\theta_{2}+1}{2}}\}\leq|(1-s^{2})^{\frac{\theta_{2}+1}{2}}|\leq\max\{1,2^{-\frac{\theta_{2}+1}{2}}\}$ in $[0,\frac{1}{\sqrt{2}}].$ Then it makes sense iff $\theta_{1}>-(n-1)$ and $n+\theta_{1}+\theta_{2}>0$.

It remains to analyze the last term $I_{3}$. Note that $|x'|\geq\sqrt{r^{2}-|x'|^{2}}\geq x_{n}$ if $\frac{1}{\sqrt{2}}r\leq|x'|\leq r$ and $0\leq x_{n}\leq\sqrt{r^{2}-|x'|^{2}}$. Then we deduce
\begin{align*}
I_{3}\sim&\int_{B_{r}'\setminus B'_{\frac{1}{\sqrt{2}}r}}|x'|^{\theta_{1}}dx'\int^{\sqrt{r^{2}-|x'|^{2}}}_{0}|x'|^{\theta_{2}}dx_{n}\notag\\
\sim&\int_{\frac{1}{\sqrt{2}}r}^{r}s^{n+\theta_{1}+\theta_{2}-2}\sqrt{r^{2}-s^{2}}ds\sim r^{n+\theta_{1}+\theta_{2}}\int_{\frac{1}{\sqrt{2}}}^{1}s^{n+\theta_{1}+\theta_{2}-2}\sqrt{1-s^{2}}ds\notag\\
\sim&r^{n+\theta_{1}+\theta_{2}},
\end{align*}
where we used the fact that the integrand $s^{n+\theta_{1}+\theta_{2}}\sqrt{1-s^{2}}$ has no singular point in $[\frac{1}{\sqrt{2}},1].$ Consequently, combining these above facts, we obtain that if $(\theta_{1},\theta_{2})\in\mathcal{B}$, then $d\mu$ is a Radon measure. The proof is complete.

\end{proof}

\begin{definition}
A Radon measure $d\mu$ is called doubling if there exists some constant $0<C<\infty$ such that $\mu(B_{2R}(\bar{x}))\leq C\mu(B_{R}(\bar{x}))$ for any $\bar{x}\in\mathbb{R}^{n}$ and $R>0$.
\end{definition}
\begin{theorem}\label{Lem009}
The Radon measure $d\mu=wdx$ is doubling if $(\theta_{1},\theta_{2})\in\mathcal{A}\cup\mathcal{B}$.
\end{theorem}
\begin{remark}
$w=|x'|^{\theta_{1}}|x|^{\theta_{2}}$ degenerates to be an isotropic weight if $\theta_{1}=0$. In this case, it is doubling if $\theta_{2}>-n$. Its proof is simple and direct, see pages 505-506 in \cite{G2014} for more details. By contrast, it will involve complex analyses, computations and discussions if $\theta_{1}\neq0$.
\end{remark}

\begin{proof}
For any $\bar{x}\in\mathbb{R}^{n}$ and $R>0$, we divide all balls $B_{R}(\bar{x})$ into two types as follows: the first type satisfies $|\bar{x}|\geq3R$ and the second type satisfies $|\bar{x}|<3R$.
\noindent{\bf Step 1.}
Consider the case when $|\bar{x}|\geq 3R$. Observe that
\begin{align}\label{DMZQM001}
\int_{B_{R}(\bar{x})}|x'|^{\theta_{1}}|x|^{\theta_{2}}dx\geq
\begin{cases}
(|\bar{x}|-R)^{\theta_{2}}\int_{B_{R}(\bar{x})}|x'|^{\theta_{1}}dx,&\text{if }\theta_{2}\geq0,\\
(|\bar{x}|+R)^{\theta_{2}}\int_{B_{R}(\bar{x})}|x'|^{\theta_{1}}dx,&\text{if }\theta_{2}<0,
\end{cases}
\end{align}
and
\begin{align}\label{DMZQM002}
\int_{B_{2R}(\bar{x})}|x'|^{\theta_{1}}|x|^{\theta_{2}}dx\leq
\begin{cases}
(|\bar{x}|+2R)^{\theta_{2}}\int_{B_{2R}(\bar{x})}|x'|^{\theta_{1}}dx,&\text{if }\theta_{2}\geq0,\\
(|\bar{x}|-2R)^{\theta_{2}}\int_{B_{2R}(\bar{x})}|x'|^{\theta_{1}}dx,&\text{if }\theta_{2}<0.
\end{cases}
\end{align}
On one hand, we have
\begin{align*}
\int_{B_{R}(\bar{x})}|x'|^{\theta_{1}}dx=&2\int_{B_{R}'(\bar{x}')}|x'|^{\theta_{1}}\sqrt{R^{2}-|x'-\bar{x}'|^{2}}dx'\geq\sqrt{3}R\int_{B'_{R/2}(\bar{x}')}|x'|^{\theta_{1}}dx'.
\end{align*}
Observe that $|x'|^{\theta_{1}}$ increases radially if $\theta_{1}\geq0$, while it decreases radially for $\theta_{1}<0$. Then we obtain that

$(i)$ for $|\bar{x}'|\geq\frac{3}{2}R$, then
\begin{align*}
\int_{B'_{R/2}(\bar{x}')}|x'|^{\theta_{1}}dx'\geq&\omega_{n-1}\Big(\frac{R}{2}\Big)^{n-1}
\begin{cases}
\big(|\bar{x}'|-R/2\big)^{\theta_{1}},&\text{if }\theta_{1}\geq0,\\
\big(|\bar{x}'|+R/2\big)^{\theta_{1}},&\text{if }\theta_{1}<0;
\end{cases}
\end{align*}

$(ii)$ for $|\bar{x}'|<\frac{3}{2}R,$ then
\begin{align*}
\int_{B'_{R/2}(\bar{x}')}|x'|^{\theta_{1}}dx'\geq&
\begin{cases}
\int_{B'_{R/2}(0')}|x'|^{\theta_{1}}dx',&\text{if }\theta_{1}\geq0,\vspace{0.3ex}\\
\int_{B'_{R/2}(\frac{3}{2}R\frac{\bar{x}'}{|\bar{x}'|})}|x'|^{\theta_{1}}dx',&\text{if }\theta_{1}<0
\end{cases}\notag\\
\geq&\omega_{n-1}R^{n-1+\theta_{1}}
\begin{cases}
\frac{n-1}{2^{n-1+\theta_{1}}(n-1+\theta_{1})},&\text{if }\theta_{1}\geq0,\vspace{0.3ex}\\
\frac{1}{2^{n-1}},&\text{if }\theta_{1}<0.
\end{cases}
\end{align*}

On the other hand, we have
\begin{align}\label{QDWZ010}
\int_{B_{2R}(\bar{x})}|x'|^{\theta_{1}}dx=&2\int_{B_{2R}'(\bar{x}')}|x'|^{\theta_{1}}\sqrt{4R^{2}-|x'-\bar{x}'|^{2}}dx'\notag\\
\leq&4R\int_{B'_{2R}(\bar{x}')}|x'|^{\theta_{1}}dx'.
\end{align}
By the same argument as above, we deduce that

$(1)$ for $|\bar{x}'|\geq3R$, then
\begin{align}\label{QDWZ012}
\int_{B'_{2R}(\bar{x}')}|x'|^{\theta_{1}}dx'\leq&\omega_{n-1}(2R)^{n-1}
\begin{cases}
\big(|\bar{x}'|+2R\big)^{\theta_{1}},&\text{if }\theta_{1}\geq0,\\
\big(|\bar{x}'|-2R\big)^{\theta_{1}},&\text{if }\theta_{1}<0;
\end{cases}
\end{align}

$(2)$ for $|\bar{x}'|<3R$, then
\begin{align}\label{QDWZ013}
\int_{B'_{2R}(\bar{x}')}|x'|^{\theta_{1}}dx'\leq&\int_{B'_{5R}(0')}|x'|^{\theta_{1}}dx'=\frac{(n-1)\omega_{n-1}}{n-1+\theta_{1}}(5R)^{n-1+\theta_{1}}.
\end{align}
Note that if $\theta_{1}\geq0$, then
\begin{align*}
\begin{cases}
|\bar{x}'|+2R\leq2(|\bar{x}'|-R/2),&\text{for }|\bar{x}'|\geq3R,\\
R<|\bar{x}'|-R/2<5R/2,&\text{for }3R/2<|\bar{x}'|<3R,
\end{cases}
\end{align*}
while, if $\theta_{1}<0$,
\begin{align*}
\begin{cases}
|\bar{x}'|-2R\geq\frac{2}{7}(|\bar{x}'|+R/2),&\text{for }|\bar{x}'|\geq3R,\\
2R<|\bar{x}'|+R/2<7R/2,&\text{for }3R/2<|\bar{x}'|<3R.
\end{cases}
\end{align*}
Then combining these above facts, we obtain
\begin{align}\label{DMZQM003}
\int_{B_{2R}(\bar{x})}|x'|^{\theta_{1}}dx\leq C(n,\theta_{1},\theta_{2})\int_{B_{R}(\bar{x})}|x'|^{\theta_{1}}dx.
\end{align}
Since $|\bar{x}|\geq3R$, then $|\bar{x}|+2R\leq4(|\bar{x}-R|)$ and $|\bar{x}|+R\leq4(|\bar{x}|-2R)$. This, in combination with \eqref{DMZQM001}--\eqref{DMZQM003}, reads that for $|\bar{x}|\geq3R,$
\begin{align}\label{MWQAE001}
\int_{B_{2R}(\bar{x})}|x'|^{\theta_{1}}|x|^{\theta_{2}}dx\leq C(n,\theta_{1},\theta_{2})\int_{B_{R}(\bar{x})}|x'|^{\theta_{1}}|x|^{\theta_{2}}dx.
\end{align}

\noindent{\bf Step 2.}
Let $|\bar{x}|<3R$. Then we have
\begin{align*}
\int_{B_{2R}(\bar{x})}|x'|^{\theta_{1}}|x|^{\theta_{2}}dx\leq \int_{B_{5R}(0)}|x'|^{\theta_{1}}|x|^{\theta_{2}}dx\leq C(n,\theta_{1},\theta_{2})R^{n+\theta_{1}+\theta_{2}}.
\end{align*}
First, if $\theta_{1}<0$, then
\begin{align*}
\int_{B_{R}(\bar{x})}|x'|^{\theta_{1}}|x|^{\theta_{2}}\geq&\int_{B_{R}(\bar{x})}|x|^{\theta_{1}+\theta_{2}}dx\notag\\
\geq&
\begin{cases}
\int_{B_{R}(0)}|x|^{\theta_{1}+\theta_{2}}dx,&\text{if }\theta_{2}\geq-\theta_{1},\\
\int_{B_{R}(3R\frac{\bar{x}}{|\bar{x}|})}|x|^{\theta_{1}+\theta_{2}}dx,&\text{if } \theta_{2}<-\theta_{1}
\end{cases}\notag\\
\geq&R^{n+\theta_{1}+\theta_{2}}
\begin{cases}
\frac{n\omega_{n}}{n+\theta_{1}+\theta_{2}},&\text{if }\theta_{2}\geq-\theta_{1},\\
2^{\theta_{1}+\theta_{2}}\omega_{n},&\text{if } \theta_{2}<-\theta_{1},
\end{cases}
\end{align*}
where we utilized the fact that $|x|^{\theta_{1}+\theta_{2}}$ is radially increasing if $\theta_{1}+\theta_{2}\geq0$ and radially decreasing if $\theta_{1}+\theta_{2}<0$.

Second, if $\theta_{1}\geq0$, we discuss as follows:

$(i)$ for $\theta_{2}\geq0$, similarly as before, we have
\begin{align*}
\int_{B_{R}(\bar{x})}|x'|^{\theta_{1}}|x|^{\theta_{2}}dx\geq&\int_{B_{R}(\bar{x})}|x'|^{\theta_{1}+\theta_{2}}dx\geq\sqrt{3}R\int_{B'_{R}(\bar{x}')}|x'|^{\theta_{1}+\theta_{2}}dx'\notag\\
\geq&\int_{B_{R}'(0')}|x'|^{\theta_{1}+\theta_{2}}dx'=\frac{(n-1)\omega_{n-1}}{n+\theta_{1}+\theta_{2}-1}R^{n+\theta_{1}+\theta_{2}};
\end{align*}

$(ii)$ for $\theta_{2}<0$, then
\begin{align*}
\int_{B_{R}(\bar{x})}|x'|^{\theta_{1}}|x|^{\theta_{2}}dx=&\int_{B_{R}(\bar{x})}|x'|^{\theta_{1}+\theta_{2}}\Big(\frac{|x'|}{|x|}\Big)^{-\theta_{2}}dx\notag\\
\geq&\frac{8^{\theta_{2}}}{2^{\theta_{1}+\theta_{2}}}R^{\theta_{1}+\theta_{2}}\int_{B_{R}(\bar{x})\cap\{|x'|>R/2\}}dx\notag\\
\geq&\frac{8^{\theta_{2}}}{2^{\theta_{1}+\theta_{2}}}R^{\theta_{1}+\theta_{2}}\int_{B_{R}(0)\cap\{|x'|>R/2\}}dx\notag\\
\geq&\frac{8^{\theta_{2}}(\omega_{n}-2^{2-n}\omega_{n-1})}{2^{\theta_{1}+\theta_{2}}}R^{n+\theta_{1}+\theta_{2}}.
\end{align*}
Then combining these aforementioned facts, we obtain that \eqref{MWQAE001} also holds if $|\bar{x}|<3R.$ The proof is complete.

\end{proof}

\begin{definition}
Let $1<q<\infty$. We say that $w$ is an $A_{q}$-weight, if there is a positive constant $C=C(n,q,w)$ such that
\begin{align*}
\dashint_{B}wdx\left(\dashint_{B}w^{-\frac{1}{q-1}}dx\right)^{q-1}\leq C(n,q,w),\quad\mathrm{with}\;\dashint_{B}=\frac{1}{|B|}\int_{B},
\end{align*}
for any ball $B$ in $\mathbb{R}^{n}$.

\end{definition}

\begin{theorem}\label{THM000Z}
Let $1<q<\infty$. If $(\theta_{1},\theta_{2})\in(\mathcal{A}\cup\mathcal{B})\cap(\mathcal{C}_{q}\cup\mathcal{D}_{q})$, then $w=|x'|^{\theta_{1}}|x|^{\theta_{2}}$ is an $A_{q}$-weight.
\end{theorem}
\begin{remark}
From Theorems \ref{Lem009} and \ref{THM000Z}, we see that the Radon measure $d\mu=wdx=|x'|^{\theta_{1}}|x|^{\theta_{2}}dx$ is doubling on a larger range $(\theta_{1},\theta_{2})\in\mathcal{A}\cup\mathcal{B}$. This implies that when $(\theta_{1},\theta_{2})\in(\mathcal{A}\cup\mathcal{B})\setminus(\mathcal{C}_{q}\cup\mathcal{D}_{q})$, the weight $w=|x'|^{\theta_{1}}|x|^{\theta_{2}}$ provides an example of a doubling measure but is not in $A_{q}$.

\end{remark}

\begin{proof}
For $1<q<\infty,$ according to the definition of $A_{q}$-weight, it needs to verify the following inequality:
\begin{align}\label{MDE016}
\dashint_{B}wdx\left(\dashint_{B}w^{-\frac{1}{q-1}}dx\right)^{q-1}\leq C(n,q,\theta_{1},\theta_{2}),
\end{align}
for any ball $B\subset\mathbb{R}^{n}$. For any $R>0$ and $\bar{x}\in\mathbb{R}^{n}$, the ball $B_{R}(\bar{x})$ must belong to one of the following two types: $|\bar{x}|\geq3R$ and $|\bar{x}|<3R$. On one hand, if $|\bar{x}|\geq3R$, then we have
\begin{align}\label{QMDA010}
\frac{2}{3}|\bar{x}|\leq|\bar{x}|-R\leq|x|\leq|\bar{x}|+R\leq\frac{4}{3}|\bar{x}|, \quad\text{for }x\in B_{R}(\bar{x}).
\end{align}
Applying \eqref{DMZQM002}--\eqref{QDWZ013} with $B_{2R}(\bar{x})$ and $B'_{2R}(\bar{x}')$ replaced by $B_{R}(\bar{x})$ and $B'_{R}(\bar{x}')$, it follows from \eqref{QMDA010} that
\begin{align*}
\int_{B_{R}(\bar{x})}|x'|^{\theta_{1}}|x|^{\theta_{2}}dx\leq&C(\theta_{2})|\bar{x}|^{\theta_{2}}\int_{B_{R}(\bar{x})}|x'|^{\theta_{1}}dx\leq C(\theta_{2})|\bar{x}|^{\theta_{2}+1}\int_{B_{R}'(\bar{x}')}|x'|^{\theta_{1}}dx'\notag\\
\leq&C(n,\theta_{1},\theta_{2})|\bar{x}|^{n+\theta_{1}+\theta_{2}},
\end{align*}
and
\begin{align*}
\int_{B_{R}(\bar{x})}|x'|^{-\frac{\theta_{1}}{q-1}}|x|^{-\frac{\theta_{2}}{q-1}}dx\leq&C(q,\theta_{2})|\bar{x}|^{-\frac{\theta_{2}}{q-1}}\int_{B_{R}(\bar{x})}|x'|^{-\frac{\theta_{1}}{q-1}}dx\notag\\
\leq& C(q,\theta_{2})|\bar{x}|^{-\frac{\theta_{2}}{q-1}+1}\int_{B_{R}'(\bar{x}')}|x'|^{-\frac{\theta_{1}}{q-1}}dx'\notag\\
\leq& C(n,q,\theta_{1},\theta_{2})|\bar{x}|^{n-\frac{\theta_{1}+\theta_{2}}{p-1}},
\end{align*}
where we require that $-(n-1)<\theta_{1}<(n-1)(q-1)$ and $\theta_{2}\in\mathbb{R}$. Combining these two relations, we obtain that \eqref{MDE016} holds in the case of $|\bar{x}|\geq3R$.

On the other hand, if $|\bar{x}|<3R$, we have $|x|\leq4R$ for $x\in B_{R}(\bar{x})$. Therefore, it follows from Lemma \ref{MWQAZ090} that
\begin{align*}
\dashint_{B_{R}(\bar{x})}|x'|^{\theta_{1}}|x|^{\theta_{2}}dx\leq 4^{n}\dashint_{B_{4R}(0)}|x'|^{\theta_{1}}|x|^{\theta_{2}}dx\leq C(n,\theta_{1},\theta_{2})R^{\theta_{1}+\theta_{2}},
\end{align*}
and
\begin{align*}
\left(\dashint_{B_{R}(\bar{x})}|x'|^{-\frac{\theta_{1}}{q-1}}|x|^{-\frac{\theta_{2}}{q-1}}dx\right)^{q-1}\leq& 4^{n(q-1)}\left(\dashint_{B_{4R}(0)}|x'|^{-\frac{\theta_{1}}{q-1}}|x|^{-\frac{\theta_{2}}{q-1}}dx\right)^{q-1}\notag\\
\leq& C(n,q,\theta_{1},\theta_{2})R^{-\theta_{1}-\theta_{2}},
\end{align*}
where these two relations hold if $(\theta_{1},\theta_{2}),(-\frac{\theta_{1}}{q-1},-\frac{\theta_{2}}{q-1})\in\mathcal{A}\cup\mathcal{B}$, that is, $(\theta_{1},\theta_{2})\in(\mathcal{A}\cup\mathcal{B})\cap(\mathcal{C}_{q}\cup\mathcal{D}_{q})$. Therefore, \eqref{QMDA010} holds for any $B\subset\mathbb{R}^{n}$. The proof is complete.

\end{proof}

Denote $d\mu:=wdx=|x'|^{\theta_{1}}|x|^{\theta_{2}}dx.$ Combining Theorem 15.21 and Corollary 15.35 in \cite{HKM2006} and Theorem \ref{THM000Z} above, we obtain the following anisotropic weighted Poincar\'{e} inequality.
\begin{corollary}\label{QWZM090}
For $n\geq2$ and $1<q<\infty$, let $(\theta_{1},\theta_{2})\in[(\mathcal{A}\cup\mathcal{B})\cap(C_{q}\cup\mathcal{D}_{q})]\cup\{\theta_{1}=0,\,\theta_{2}>-n\}$. Then for any $B:=B_{R}(\bar{x})\subset\mathbb{R}^{n}$, $R>0$, and $\varphi\in W^{1,q}(B,w)$,
\begin{align}\label{DQ003}
\int_{B}|\varphi-\varphi_{B}|^{q}d\mu\leq C(n,q,\theta_{1},\theta_{2})R^{q}\int_{B}|\nabla\varphi|^{q}d\mu,
\end{align}
where $\varphi_{B}=\frac{1}{\mu(B)}\int_{B}\varphi d\mu.$
\end{corollary}
\begin{remark}
It is worth emphasizing that according to Corollary 15.35 in \cite{HKM2006}, \eqref{DQ003} holds for any $(\theta_{1},\theta_{2})\in\{\theta_{1}=0,\,\theta_{2}>-n\}$ and $1<q<\infty$. This conclusion is very strong, which is achieved by combining the theories of $A_{q}$-weights and quasiconformal mappings, see Chapter 15 of \cite{HKM2006} for further details.

\end{remark}

Making use of the anisotropic weighted Poincar\'{e} inequality in Corollary \ref{QWZM090}, we can establish the corresponding weighted isoperimetric inequality of De Giorgi type as follows.
\begin{prop}\label{prop002}
For $n\geq2$ and $1<q<\infty$, let $(\theta_{1},\theta_{2})\in[(\mathcal{A}\cup\mathcal{B})\cap(C_{q}\cup\mathcal{D}_{q})]\cup\{\theta_{1}=0,\,\theta_{2}>-n\}$. Then for any $R>0$, $l>k$ and $u\in W^{1,q}(B_{R},w)$,
\begin{align}\label{pro001}
&(l-k)^{q}\left(\int_{\{u\geq l\}\cap B_{R}}d\mu\right)^{q}\int_{\{u\leq k\}\cap B_{R}}d\mu\notag\\
&\leq C(n,q,\theta_{1},\theta_{2})R^{q(n+\theta_{1}+\theta_{2}+1)}\int_{\{k<u<l\}\cap B_{R}}|\nabla u|^{q}d\mu,
\end{align}
and
\begin{align}\label{pro002}
&(l-k)^{q}\left(\int_{\{u\leq k\}\cap B_{R}}d\mu\right)^{q}\int_{\{u\geq l\}\cap B_{R}}d\mu\notag\\
&\leq C(n,q,\theta_{1},\theta_{2})R^{q(n+\theta_{1}+\theta_{2}+1)}\int_{\{k<u<l\}\cap B_{R}}|\nabla u|^{q}d\mu,
\end{align}
where $d\mu=wdx=|x'|^{\theta_{1}}|x|^{\theta_{2}}dx.$
\end{prop}
\begin{remark}
Since the index $q>1$, we have to establish two isoperimetric inequalities in Proposition \ref{prop002}, which are used to capture the decaying rates of the distribution function in Lemmas \ref{lem005ZZW} and \ref{lem005} below. Meanwhile, it also causes more complex calculations in the proofs of Lemmas \ref{lem005ZZW} and \ref{lem005}.

\end{remark}

\begin{proof}
\noindent{\bf Step 1.}
Set
\begin{align*}
u_{1}=\inf\{u,l\}-\inf\{u,k\},\quad\bar{u}_{1}=\frac{1}{\mu(B_{R})}\int_{B_{R}}u_{1}d\mu.
\end{align*}
First, we have
\begin{align*}
\int_{\{u_{1}=0\}\cap B_{R}}\bar{u}_{1}^{q}d\mu=&\frac{1}{(\mu(B_{R}))^{q}}\left(\int_{B_{R}}u_{1}d\mu\right)^{q}\int_{\{u\leq k\}\cap B_{R}}d\mu\notag\\
\geq&\frac{C(l-k)^{q}}{R^{q(n+\theta_{1}+\theta_{2})}}\left(\int_{\{u\geq l\}\cap B_{R}}d\mu\right)^{q}\int_{\{u\leq k\}\cap B_{R}}d\mu.
\end{align*}
Second, it follows from Corollary \ref{QWZM090} that
\begin{align*}
\int_{\{u_{1}=0\}\cap B_{R}}|\bar{u}_{1}|^{q}d\mu\leq&\int_{B_{R}}|u_{1}-\bar{u}_{1}|^{q}d\mu\leq CR^{q}\int_{B_{R}}|\nabla u_{1}|^{q}d\mu\notag\\
=&CR^{q}\int_{\{k<u<l\}\cap B_{R}}|\nabla u|^{q}d\mu.
\end{align*}
The proof of \eqref{pro001} is finished.

\noindent{\bf Step 2.}
Denote
\begin{align*}
u_{2}=\sup\{u,l\}-\sup\{u,k\},\quad\bar{u}_{2}=\frac{\int_{B_{R}}u_{2}d\mu}{|B_{R}|_{\mu}}.
\end{align*}
By the same argument as before, we have
\begin{align*}
\int_{\{u_{2}=0\}\cap B_{R}}\bar{u}_{2}^{q}d\mu=&\frac{1}{(\mu(B_{R}))^{q}}\left(\int_{B_{R}}u_{2}d\mu\right)^{q}\int_{\{u\geq l\}\cap B_{R}}d\mu\notag\\
\geq&\frac{C(l-k)^{q}}{R^{q(n+\theta_{1}+\theta_{2})}}\left(\int_{\{u\leq k\}\cap B_{R}}d\mu\right)^{q}\int_{\{u\geq l\}\cap B_{R}}d\mu,
\end{align*}
and
\begin{align*}
\int_{\{u_{2}=0\}\cap B_{R}}|\bar{u}_{2}|^{q}d\mu\leq&\int_{B_{R}}|u_{2}-\bar{u}_{2}|^{q}d\mu\leq CR^{q}\int_{B_{R}}|\nabla u_{2}|^{q}d\mu\notag\\
=&CR^{q}\int_{\{k<u<l\}\cap B_{R}}|\nabla u|^{q}d\mu.
\end{align*}
The proof is complete.

\end{proof}

\section{Regularity for solutions to degenerate elliptic equations with anisotropic weights}\label{SEC003}

Throughout this section, denote $d\mu:=wdx=|x'|^{\theta_{1}}|x|^{\theta_{2}}dx.$ The first step is to establish a Caccioppoli inequality for the truncated solution.
\begin{lemma}\label{lem003ZZW}
Let $u$ be the solution of problem \eqref{PROBLEM001ZW}. Then for any nonnegative $\eta\in C^{\infty}_{0}(B_{R}(x_{0}))$ with any $B_{R}(x_{0})\subset B_{1}$,
\begin{align*}
\int_{B_{R}(x_{0})}|\nabla(v\eta)|^{p}wdx\leq C(n,p,\lambda)\int_{B_{R}(x_{0})}|\nabla\eta|^{p}|v|^{p}wdx,
\end{align*}
where $v=(u-k)^{+}\,\mathrm{or}\,(u-k)^{-}$ with $k\geq0$.

\end{lemma}

\begin{proof}
First, pick test function $\varphi=v\eta^{p}$ if $v=(u-k)^{+}$. Since
\begin{align*}
0=\int_{B_{R}(x_{0})}Aw|\nabla u|^{p-2}\nabla u\cdot\nabla\varphi=\int_{B_{R}(x_{0})}Aw|\nabla v|^{p-2}\nabla v\cdot\nabla\varphi,
\end{align*}
then it follows from Young's inequality that
\begin{align}\label{WAD002}
&\lambda\int_{B_{R}(x_{0})}|\nabla v|^{p}\eta^{p}wdx\notag\\
&\leq\int_{B_{R}(x_{0})}Aw|\nabla v|^{p}\eta^{p}=-p\int_{B_{R}(x_{0})}Aw|\nabla v|^{p-2}\nabla v\cdot\nabla\eta v\eta^{p-1}\notag\\
&\leq p\lambda\int_{B_{R}(x_{0})}|\eta\nabla v|^{p-1}|v\nabla\eta|wdx\notag\\
&\leq\frac{\lambda}{2}\int_{B_{R}(x_{0})}|\nabla v|^{p}\eta^{p}wdx+C\int_{B_{R}(x_{0})}|v|^{p}|\nabla\eta|^{p}wdx,
\end{align}
which yields that
\begin{align*}
\int_{B_{R}(x_{0})}|\nabla(v\eta)|^{p}wdx\leq&2^{p-1}\int_{B_{R}(x_{0})}\big(|\eta\nabla v|^{p}+|v\nabla\eta|^{p}\big)wdx\notag\\
\leq& C\int_{B_{R}(x_{0})}|\nabla\eta|^{p}|v|^{p}wdx.
\end{align*}

Second, choose test function $\varphi=-v\eta^{2}$ if $v=(u-k)^{-}$. Then we have
\begin{align*}
0=\int_{B_{R}(x_{0})}Aw|\nabla u|^{p-2}\nabla u\cdot\nabla\varphi=-\int_{B_{R}(x_{0})}Aw|\nabla v|^{p-2}\nabla v\cdot\nabla\varphi.
\end{align*}
Therefore, in exactly the same way to \eqref{WAD002}, we obtain that Lemma \ref{lem003ZZW} holds.

\end{proof}

We now improve the oscillation of the solution $u$ in a small domain provided that $u$ is small on a large portion of a larger domain.
\begin{lemma}\label{LEM0035ZZW}
Assume that $n\geq2$, $(\theta_{1},\theta_{2})\in(\mathcal{A}\cup\mathcal{B})\cap\mathcal{F}$, $1<p<n+\theta_{1}+\theta_{2}$. For $R\in(0,1)$, let $0\leq m\leq\inf\limits_{B_{R}}u\leq\sup\limits_{B_{R}}u\leq M\leq\overline{M}$. Then there exists a small constant $0<\tau_{0}=\tau_{0}(n,p,\theta_{1},\theta_{2},\lambda)<1$ such that for any $\varepsilon>0$ and $0<\tau<\tau_{0}$,

$(a)$ if
\begin{align}\label{ZWZ007ZZW}
|\{x\in B_{R}: u>M-\varepsilon\}|_{\mu}\leq\tau|B_{R}|_{\mu},
\end{align}
then
\begin{align}\label{DZ001ZZW}
u\leq M-\frac{\varepsilon}{2},\quad\text{for $x\in B_{R/2}$;}
\end{align}

$(b)$ if
\begin{align*}
|\{x\in B_{R}: u<m+\varepsilon\}|_{\mu}\leq\tau|B_{R}|_{\mu},
\end{align*}
then
\begin{align}\label{ZWZ009ZZW}
u\geq m-\frac{\varepsilon}{2},\quad\text{for $x\in B_{R/2}$.}
\end{align}

\end{lemma}
\begin{remark}
The assumed condition in \eqref{ZWZ007ZZW} is natural, since
\begin{align*}
|\{x\in B_{R}:\,u>M-\varepsilon\}|_{\mu}\rightarrow0,\quad\mathrm{as}\; \varepsilon\rightarrow0.
\end{align*}
This fact also implies that the value of $\tau$ can be chosen to satisfy that $\tau\rightarrow0$, as $\varepsilon\rightarrow0$. Then the key to applying Lemma \ref{LEM0035ZZW} lies in making clear the dependency between $\tau$ and $\varepsilon$ in condition \eqref{ZWZ007ZZW}. The purpose will be achieved by establishing the explicit decaying estimates in terms of the distribution function of $u$ in Lemma \ref{lem005ZZW} below.
\end{remark}

\begin{proof}
\noindent{\bf Step 1.}
For $\varepsilon>0$ and $i=0,1,2,...,$ let
\begin{align*}
r_{i}=\frac{R}{2}+\frac{R}{2^{i+1}},\quad k_{i}=M-\varepsilon+\frac{\varepsilon}{2}(1-2^{-i}).
\end{align*}
Take a cut-off function $\eta_{i}\in C_{0}^{\infty}(B_{r_{i}})$ satisfying that $\eta_{i}=1$ in $B_{r_{i+1}}$, $0\leq\eta_{i}\leq1$, $|\nabla\eta_{i}|\leq C(r_{i}-r_{i+1})^{-1}$ in $B_{r_{i}}$. For $k\in[m,M]$ and $\rho\in(0,R]$, write $v_{i}=(u-k_{i})^{+}$ and $A(k,\rho)=\{x\in B_{\rho}: u>k\}$. By Theorem 1.1 in \cite{LY2021}, we have the following anisotropic Caffarelli-Kohn-Nirenberg type inequality:
\begin{align*}
\|u\|_{L^{\frac{(n+\theta_{1}+\theta_{2})p}{n+\theta_{1}+\theta_{2}-p}}(B_{R},w)}\leq C(n,p,\theta_{1},\theta_{2})\|\nabla u\|_{L^{p}(B_{R},w)},\quad \forall\;u\in W_{0}^{1,p}(B_{R},w),
\end{align*}
which, together with Lemma \ref{lem003ZZW}, reads that
\begin{align*}
\left(\int_{B_{R}}|\eta_{i}v_{i}|^{p\chi}wdx\right)^{\frac{1}{\chi}}\leq C\int_{B_{R}}|v_{i}|^{p}|\nabla\eta_{i}|^{p}wdx,\quad\chi=\frac{n+\theta_{1}+\theta_{2}}{n+\theta_{1}+\theta_{2}-p}.
\end{align*}
Since
\begin{align*}
\int_{B_{R}}|v_{i}|^{p}|\nabla\eta_{i}|^{p}wdx\leq\frac{C(M-k_{i})^{p}}{(r_{i}-r_{i+1})^{p}}|A(k_{i},r_{i})|_{\mu},
\end{align*}
and
\begin{align*}
\left(\int_{B_{R}}|\eta_{i}v_{i}|^{p\chi}wdx\right)^{\frac{1}{\chi}}\geq(k_{i+1}-k_{i})^{p}|A(k_{i+1},r_{i+1})|_{\mu}^{\frac{1}{\chi}},
\end{align*}
it then follows that there exists a positive constant $i_{0}=i_{0}(n,p,\theta_{1},\theta_{2})>0$ such that for $i\geq i_{0}$,
\begin{align*}
\frac{|A(k_{i+1},r_{i+1})|_{\mu}}{|B_{R}|_{\mu}}\leq&(C2^{2p(i+2)})^{\chi}R^{(\theta_{1}+\theta_{2}+n)(\chi-1)-p\chi}\left(\frac{|A(k_{i},r_{i})|_{\mu}}{|B_{R}|_{\mu}}\right)^{\chi}\notag\\
=&(C2^{2p(i+2)})^{\chi}\left(\frac{|A(k_{i},r_{i})|_{\mu}}{|B_{R}|_{\mu}}\right)^{\chi}\notag\\
\leq&\prod\limits^{i}_{s=0}(C2^{2p(i+2-s)})^{\chi^{s+1}}\left(\frac{|A(k_{0},r_{0})|_{\mu}}{|B_{R}|_{\mu}}\right)^{\chi^{i+1}}\notag\\
\leq&(C^{\ast})^{i+1}\left(\frac{|A(k_{0},r_{0})|_{\mu}}{|Q_{R}|_{\mu}}\right)^{\chi(i+1)}.
\end{align*}
Fix $\tau_{0}=(C^{\ast})^{-\chi}$. Then we deduce that for any $\varepsilon>0$ and $0<\tau<\tau_{0}$, if \eqref{ZWZ007ZZW} holds, then
\begin{align*}
\frac{|A(k_{i+1},r_{i+1})|_{\mu}}{|B_{R}|_{\mu}}\leq (C^{\ast}\tau^{\chi})^{i+1}=\left(\frac{\tau}{\tau_{0}}\right)^{\chi(i+1)}\rightarrow0,\quad\mathrm{as}\;i\rightarrow\infty.
\end{align*}
Hence, \eqref{DZ001ZZW} is proved.

\noindent{\bf Step 2.} Similarly as above, set
\begin{align*}
r_{i}=\frac{R}{2}+\frac{R}{2^{i+1}},\quad \tilde{k}_{i}=m+\varepsilon-\frac{\varepsilon}{2}(1-2^{-i}),\quad i\geq0.
\end{align*}
For $k\in[m,M]$ and $\rho\in(0,R]$, let $\tilde{v}_{i}=(u-\tilde{k}_{i})^{-}$ and $\tilde{A}(k,\rho)=\{x\in B_{\rho}: u<k\}$. Then applying the proof of \eqref{DZ001ZZW} with minor modification, we obtain that \eqref{ZWZ009ZZW} also holds. The proof is finished.

\end{proof}

The decaying estimates for the distribution function of the solution $u$ are established as follows.
\begin{lemma}\label{lem005ZZW}
Suppose that $n\geq2$, $(\theta_{1},\theta_{2})\in[(\mathcal{A}\cup\mathcal{B})\cap(C_{q}\cup\mathcal{D}_{q})\cap\mathcal{F}]\cup\{\theta_{1}=0,\theta_{2}>-(n-1)\}$, $1<q<p<n+\theta_{1}+\theta_{2}$, $0<\gamma<1$, $0<R<\frac{1}{2}$, and $0\leq m\leq\inf\limits_{B_{2R}}u\leq\sup\limits_{B_{2R}}u\leq M$. Then for any $\varepsilon>0$,

$(a)$ if
\begin{align}\label{VD008}
\frac{|\{x\in B_{R}:u>M-\varepsilon\}|_{\mu}}{|B_{R}|_{\mu}}\leq1-\gamma,
\end{align}
then for any $j\geq1$,
\begin{align}\label{DEC001ZZW}
\frac{|\{x\in B_{R}:u>M-\frac{\varepsilon}{2^{j}}\}|_{\mu}}{|B_{R}|_{\mu}}\leq\frac{C}{\sqrt[q]{\gamma}j^{\frac{p-q}{pq}}};
\end{align}

$(b)$ if
\begin{align}\label{VDC006}
\frac{|\{x\in B_{R}:u<m+\varepsilon\}|_{\mu}}{|B_{R}|_{\mu}}\leq1-\gamma,
\end{align}
then for any $j\geq1$,
\begin{align*}
\frac{|\{x\in B_{R}:u<m+\frac{\varepsilon}{2^{j}}\}|_{\mu}}{|B_{R}|_{\mu}}\leq\frac{C}{\sqrt[q]{\gamma}j^{\frac{p-q}{pq}}},
\end{align*}
where $C=C(n,p,q,\theta_{1},\theta_{2},\lambda)$.

\end{lemma}

\begin{proof}
\noindent{\bf Step 1.}
For $i\geq0$, let $k_{i}=M-\frac{\varepsilon}{2^{i}}$ and $A(k_{i},R)=B_{R}\cap\{u>k_{i}\}.$ From \eqref{pro001}, we know that for $q>1,$
\begin{align}\label{WMZ001ZZW}
&(k_{i+1}-k_{i})^{q}|A(k_{i+1},R)|_{\mu}^{q}|B_{R}\setminus A(k_{i},R)|_{\mu}\notag\\
&\leq CR^{q(n+\theta_{1}+\theta_{2}+1)}\int_{A(k_{i},R)\setminus A(k_{i+1},R)}|\nabla u|^{q}wdx.
\end{align}
Using \eqref{VD008}, we have
\begin{align*}
|B_{R}\setminus A(k_{i},R)|_{\mu}\geq\gamma|B_{R}|_{\mu}=C(n,\theta_{1},\theta_{2})\gamma R^{n+\theta_{1}+\theta_{2}}.
\end{align*}
This, together with \eqref{WMZ001ZZW}, shows that
\begin{align*}
&|A(k_{i+1},R)|_{\mu}\leq\frac{C2^{i+1}}{\varepsilon\sqrt[q]{\gamma}}R^{\frac{(n+\theta_{1}+\theta_{2})(q-1)}{q}+1}\left(\int_{A(k_{i},R)\setminus A(k_{i+1},R)}|\nabla u|^{q}wdx\right)^{\frac{1}{q}}.
\end{align*}
Since $1<q<p<n+\theta_{1}+\theta_{2}$, we then have from H\"{o}lder's inequality that
\begin{align*}
&\int_{A(k_{i},R)\setminus A(k_{i+1},R)}|\nabla u|^{q}wdx\notag\\
&\leq\left(\int_{A(k_{i},R)\setminus A(k_{i+1},R)}|\nabla u|^{p}wdx\right)^{\frac{q}{p}}\left(\int_{A(k_{i},R)\setminus A(k_{i+1},R)}wdx\right)^{\frac{p-q}{p}}\notag\\
&\leq\left(\int_{B_{R}}|\nabla(u-k_{i})^{+}|^{p}wdx\right)^{\frac{q}{p}}\left(\int_{A(k_{i},R)\setminus A(k_{i+1},R)}wdx\right)^{\frac{p-q}{p}}.
\end{align*}
Choose a cut-off function $\eta\in C_{0}^{\infty}(B_{2R})$ satisfying that
\begin{align}\label{ETA003ZZW}
\eta=1\;\mathrm{in}\;B_{R},\quad0\leq\eta\leq1,\;|\nabla\eta|\leq\frac{C(n)}{R}\;\,\mathrm{in}\;B_{2R}.
\end{align}
It then follows from Lemma \ref{lem003ZZW} that
\begin{align*}
&\left(\int_{B_{R}}|\nabla (u-k_{i})^{+}|^{p}wdx\right)^{\frac{1}{p}}\leq C\left(\int_{B_{2R}}|(u-k_{i})^{+}|^{p}|\nabla\eta|^{p}wdx\right)^{\frac{1}{p}}\leq\frac{C\varepsilon}{2^{i}}R^{\frac{n+\theta_{1}+\theta_{2}-p}{p}}.
\end{align*}
A combination of these above facts shows that
\begin{align*}
|A(k_{i+1},R)|_{\mu}\leq&\frac{C}{\sqrt[q]{\gamma}}R^{(n+\theta_{1}+\theta_{2})(1-\frac{p-q}{pq})}|A(k_{i},R)\setminus A(k_{i+1},R)|_{\mu}^{\frac{p-q}{pq}}.
\end{align*}
This leads to that for $j\geq1$,
\begin{align*}
j|A(k_{j},R)|_{\nu}^{\frac{pq}{p-q}}\leq&\sum^{j-1}_{i=0}|A(k_{i+1},R)|_{\mu}^{\frac{pq}{p-q}}\leq\frac{C}{\gamma^{\frac{p}{p-q}}}R^{(n+\theta_{1}+\theta_{2})(\frac{pq}{p-q}-1)}|B_{R}|_{\mu}\notag\\
\leq&\frac{C}{\gamma^{\frac{p}{p-q}}}|B_{R}|_{\mu}^{\frac{pq}{p-q}}.
\end{align*}
Then \eqref{DEC001ZZW} is proved.

\noindent{\bf Step 2.}
For $i\geq0$, denote $\tilde{k}_{i}=m+\frac{\varepsilon}{2^{i}}$ and $\tilde{A}(k_{i},R)=B_{R}\cap\{u<k_{i}\}.$ In light of \eqref{pro002}, we see that for $q>1,$
\begin{align*}
&(\tilde{k}_{i}-\tilde{k}_{i+1})^{q}|\tilde{A}(\tilde{k}_{i+1},R)|_{\mu}^{q}|B_{R}\setminus \tilde{A}(\tilde{k}_{i},R)|_{\mu}\notag\\
&\leq CR^{q(n+\theta_{1}+\theta_{2}+1)}\int_{\tilde{A}(\tilde{k}_{i},R)\setminus \tilde{A}(\tilde{k}_{i+1},R)}|\nabla u|^{q}wdx.
\end{align*}
From \eqref{VDC006}, we have
\begin{align*}
|B_{R}\setminus \tilde{A}(\tilde{k}_{i},R)|_{\mu}\geq\gamma|B_{R}|_{\mu}=C(n,\theta_{1},\theta_{2})\gamma R^{n+\theta_{1}+\theta_{2}}.
\end{align*}
Hence we obtain
\begin{align*}
&|\tilde{A}(\tilde{k}_{i+1},R)|_{\mu}\leq\frac{C2^{i+1}}{\varepsilon\sqrt[q]{\gamma}}R^{\frac{(n+\theta_{1}+\theta_{2})(q-1)}{q}+1}\left(\int_{\tilde{A}(\tilde{k}_{i},R)\setminus \tilde{A}(\tilde{k}_{i+1},R)}|\nabla u|^{q}wdx\right)^{\frac{1}{q}}.
\end{align*}
Analogously as before, it follows from H\"{o}lder's inequality and Lemma \ref{lem003ZZW} that for $1<q<p<n+\theta_{1}+\theta_{2}$
\begin{align*}
&\left(\int_{\tilde{A}(\tilde{k}_{i},R)\setminus \tilde{A}(\tilde{k}_{i+1},R)}|\nabla u|^{q}wdx\right)^{\frac{1}{q}}\notag\\
&\leq\left(\int_{\tilde{A}(\tilde{k}_{i},R)\setminus \tilde{A}(\tilde{k}_{i+1},R)}|\nabla u|^{p}wdx\right)^{\frac{1}{p}}\left(\int_{\tilde{A}(\tilde{k}_{i},R)\setminus \tilde{A}(\tilde{k}_{i+1},R)}wdx\right)^{\frac{p-q}{pq}}\notag\\
&\leq\left(\int_{B_{R}}|\nabla(u-\tilde{k}_{i})^{-}|^{p}wdx\right)^{\frac{1}{p}}\left(\int_{\tilde{A}(\tilde{k}_{i},R)\setminus \tilde{A}(\tilde{k}_{i+1},R)}wdx\right)^{\frac{p-q}{pq}}\notag\\
&\leq C\left(\int_{B_{2R}}|(u-\tilde{k}_{i})^{-}|^{p}|\nabla\eta|^{p}wdx\right)^{\frac{1}{p}}|\tilde{A}(\tilde{k}_{i},R)\setminus \tilde{A}(\tilde{k}_{i+1},R)|_{\mu}^{\frac{p-q}{pq}}\notag\\
&\leq\frac{C\varepsilon}{2^{i}}R^{\frac{n+\theta_{1}+\theta_{2}-p}{p}}|\tilde{A}(\tilde{k}_{i},R)\setminus \tilde{A}(\tilde{k}_{i+1},R)|_{\mu}^{\frac{p-q}{pq}},
\end{align*}
where $\eta$ is given by \eqref{ETA003ZZW}. Then we obtain
\begin{align*}
|\tilde{A}(\tilde{k}_{i+1},R)|_{\mu}\leq&\frac{C}{\sqrt[q]{\gamma}}R^{(n+\theta_{1}+\theta_{2})(1-\frac{p-q}{pq})}|\tilde{A}(\tilde{k}_{i},R)\setminus \tilde{A}(\tilde{k}_{i+1},R)|_{\mu}^{\frac{p-q}{pq}},
\end{align*}
and thus,
\begin{align*}
j|\tilde{A}(\tilde{k}_{j},R)|_{\nu}^{\frac{pq}{p-q}}\leq&\sum^{j-1}_{i=0}|\tilde{A}(\tilde{k}_{i+1},R)|_{\mu}^{\frac{pq}{p-q}}\leq\frac{C}{\gamma^{\frac{p}{p-q}}}R^{(n+\theta_{1}+\theta_{2})(\frac{pq}{p-q}-1)}|B_{R}|_{\mu}\notag\\
\leq&\frac{C}{\gamma^{\frac{p}{p-q}}}|B_{R}|_{\mu}^{\frac{pq}{p-q}},\quad \text{for }j\geq1.
\end{align*}
The proof is complete.

\end{proof}

A combination of Lemmas \ref{LEM0035ZZW} and \ref{lem005ZZW} yields the following improvement on oscillation of $u$ in a small domain.
\begin{corollary}\label{AMZW01ZZW}
Assume that $n\geq2$, $(\theta_{1},\theta_{2})\in[(\mathcal{A}\cup\mathcal{B})\cap(C_{q}\cup\mathcal{D}_{q})\cap\mathcal{F}]\cup\{\theta_{1}=0,\,\theta_{2}>-(n-1)\}$, $1<q<p<n+\theta_{1}+\theta_{2}$, $0<\gamma<1$, $0<R<\frac{1}{2}$, and $0\leq m\leq\inf\limits_{B_{2R}}u\leq\sup\limits_{B_{2R}}u\leq M$. Then there exists a large constant $k_{0}>1$ depending only on $n,p,q,\theta_{1},\theta_{2},\lambda,\gamma$ such that for any $\varepsilon>0$,

$(i)$ if
\begin{align*}
\frac{|\{x\in B_{R}: u>M-\varepsilon\}|_{\mu}}{|B_{R}|_{\mu}}\leq1-\gamma,
\end{align*}
then
\begin{align*}
\sup\limits_{B_{R/2}}u\leq M-\frac{\varepsilon}{2^{k_{0}}};
\end{align*}

$(ii)$ if
\begin{align*}
\frac{|\{x\in B_{R}: u<m+\varepsilon\}|_{\mu}}{|B_{R}|_{\mu}}\leq1-\gamma,
\end{align*}
then
\begin{align*}
\inf\limits_{B_{R/2}}u\geq m+\frac{\varepsilon}{2^{k_{0}}}.
\end{align*}

\end{corollary}

\begin{proof}
Applying Lemmas \ref{LEM0035ZZW} and \ref{lem005ZZW}, we obtain that Corollary \ref{AMZW01ZZW} holds. In particular, in the case of $\theta_{1}=0,\,\theta_{2}>-(n-1)$, $1<p<n+\theta_{2}$, we fix $q=\frac{p+1}{2}$ in Lemma \ref{lem005ZZW}.

\end{proof}

We are now ready to prove Theorems \ref{THM001Z} and \ref{THM002}, respectively.
\begin{proof}[Proof of Theorem \ref{THM001Z}]
For $0<R\leq\frac{1}{2}$, denote
\begin{align*}
\overline{\mu}(R)=\sup_{B_{R}}u,\quad\underline{\mu}(R)=\inf_{B_{R}}u,\quad \omega(R)=\overline{\mu}(R)-\underline{\mu}(R).
\end{align*}
Note that one of the following two statements must hold:
either
\begin{align}\label{WAMQ001}
|\{x\in B_{R}:u>\overline{\mu}(R)-2^{-1}\omega(R)\}|_{\mu}\leq\frac{1}{2}|B_{R}|_{\mu},
\end{align}
or
\begin{align}\label{WAMQ002}
|\{x\in B_{R}:u<\underline{\mu}(R)+2^{-1}\omega(R)\}|_{\mu}\leq\frac{1}{2}|B_{R}|_{\mu}.
\end{align}
Using Corollary \ref{AMZW01ZZW} with $\gamma=\frac{1}{2}$, we derive that there is a large constant $k_{0}>1$ such that
\begin{align*}
\overline{\mu}(R/2)\leq\overline{\mu}(R)-\frac{\omega(R)}{2^{k_{0}+1}},\quad\text{when \eqref{WAMQ001} holds,}
\end{align*}
and
\begin{align*}
\underline{\mu}(R/2)\geq\underline{\mu}(R)+\frac{\omega(R)}{2^{k_{0}+1}},\quad\text{when \eqref{WAMQ002} holds.}
\end{align*}
In either case, we have
\begin{align*}
\omega(R/2)\leq\left(1-\frac{1}{2^{k_{0}+1}}\right)\omega(R)=\frac{1}{2^{\alpha}}\omega(R),\quad\mathrm{with}\;\alpha=-\frac{\ln\big(1-\frac{1}{2^{k_{0}+1}}\big)}{\ln2}.
\end{align*}
Observe that for each $0<R\leq\frac{1}{2}$, there exists an integer $l\geq1$ such that $2^{-(l+1)}<R\leq2^{-l}$. Since $\omega(R)$ increases with respect to $R$, we then have
\begin{align*}
\omega(R)\leq\omega(2^{-l})\leq 2^{-(l-1)\alpha}\omega(2^{-1})=4^{\alpha}2^{-(l+1)\alpha}\omega(2^{-1})\leq CR^{\alpha},
\end{align*}
where $C=C(n,p,q,\theta_{1},\theta_{2},\lambda,\overline{M}).$ The proof is complete.

\end{proof}

\begin{proof}[Proof of Theorem \ref{THM002}]
First, by applying the proof of Theorem \ref{THM001Z} with a slight modification, we also obtain that there exist a small constant $0<\alpha=\alpha(n,p,\theta_{2},\lambda)<1$ and a large constant $C=C(n,p,\theta_{2},\lambda,\overline{M})>0$ such that
\begin{align}\label{WMA092}
|u(x)-u(0)|\leq C|x|^{\alpha},\quad\text{for all}\;x\in B_{1/2}.
\end{align}

For $R\in(0,1/2)$, $y\in Q_{1/R}$, denote
\begin{align*}
u_{R}(y)=u(Ry),\quad A_{R}(y)=A(Ry).
\end{align*}
Hence, $u_{R}$ is the solution of
\begin{align*}
\mathrm{div}(A_{R}|y|^{\theta_{2}}\nabla u_{R})=0,\quad\mathrm{in}\;Q_{1/R}.
\end{align*}
After the change of variables, we see that this equation becomes degenerate elliptic equation in $B_{1/2}(\bar{y})$ for any $\bar{y}\in\partial B_{1}$. For any two given points $x,\tilde{x}\in B_{1/2},$ let $|\tilde{x}|\leq|x|$ without loss of generality. Denote $R=|x|$. By the interior H\"{o}lder estimate for degenerate elliptic equation, we derive that there exist two constants $0<\beta=\beta(n,p,\theta_{2},\lambda)<1$ and $0<C=C(n,p,\theta_{2},\lambda,\overline{M})$ such that for any $\bar{y}\in\partial B_{1}$,
\begin{align}\label{WAQA001ZZW}
|u_{R}(y)-u_{R}(\bar{y})|\leq C|y-\bar{y}|^{\beta},\quad\forall\;y\in B_{1/2}(\bar{y}).
\end{align}
Consequently, for $|x-\tilde{x}|\leq R^{2}$, we have from \eqref{WAQA001ZZW} that
\begin{align*}
|u(x)-u(\tilde{x})|=&\left|u_{R}(x/R)-u_{R}(\tilde{x}/R)\right|\leq C|(x-\tilde{x})/R|^{\beta}\leq C|x-\tilde{x}|^{\beta/2},
\end{align*}
while, for $|x-\tilde{x}|>R^{2}$, we deduce from \eqref{WMA092} that
\begin{align*}
|u(x)-u(\tilde{x})|\leq&|u(x)-u(0)|+|u(0)-u(\tilde{x})|\leq C\big(R^{\alpha}+|\tilde{x}|^{\alpha}\big)\leq CR^{\alpha}\leq C|x-\tilde{x}|^{\frac{\alpha}{2}}.
\end{align*}
Therefore, the proof Theorem \ref{THM002} is finished.

\end{proof}

\section{Regularity for solutions to nonlinear parabolic equations with anisotropic weights}\label{SEC00four}

Let $n\geq2$, $R>0$, and $-T\leq t_{1}<t_{2}\leq0$. For $u\in C(t_{1},t_{2};L^{2}(B_{R},w_{1}))\cap L^{2}(t_{1},t_{2};W_{0}^{1,2}(B_{R},w_{2}))$, denote
\begin{align*}
\|u\|_{V^{1}_{0}(B_{R}\times(t_{1},t_{2}))}=\sqrt{\sup\limits_{t\in(t_{1},t_{2})}\int_{B_{R}}|u|^{2}w_{1}dx+\int^{t_{2}}_{t_{1}}\int_{B_{R}}|\nabla u|^{2}w_{2}dxdt},
\end{align*}
where the anisotropic weights $w_{1}$ and $w_{2}$ are defined in \eqref{PROBLEM001}. The parabolic Sobolev inequality with anisotropic weights is now given as follows.
\begin{prop}\label{prop001}
For $n\geq2$, $R>0$, $\theta_{1}+\theta_{2}>-(n-2)$, and $-T\leq t_{1}<t_{2}\leq0$, let $u\in C(t_{1},t_{2};L^{2}(B_{R},w_{1}))\cap L^{2}(t_{1},t_{2};W_{0}^{1,2}(B_{R},w_{2}))$. Then
\begin{align*}
\|u\|_{L^{2\chi}(B_{R}\times(t_{1},t_{2}),w_{2})}\leq C\|u\|_{V^{1}_{0}(B_{R}\times(t_{1},t_{2}))},\quad\mathrm{with}\;\chi=\frac{n+\theta_{1}+\theta_{2}+2}{n+\theta_{1}+\theta_{2}},
\end{align*}
where $C=C(n,\theta_{1},\theta_{2}).$

\end{prop}
\begin{proof}
Applying the anisotropic version of the Caffarelli-Kohn-Nirenberg inequality in \cite{LY2021}, we obtain that for any $u\in W_{0}^{1,2}(B_{R},w_{2})$,
\begin{align*}
&\left(\int_{B_{R}}|u|^{\frac{2(n+\theta_{1}+\theta_{2})}{n+\theta_{1}+\theta_{2}-2}}|x'|^{\frac{\theta_{3}(n+\theta_{1}+\theta_{2})-2\theta_{1}}{n+\theta_{1}+\theta_{2}-2}}|x|^{\frac{\theta_{4}(n+\theta_{1}+\theta_{2})-2\theta_{2}}{n+\theta_{1}+\theta_{2}-2}}dx\right)^{\frac{n+\theta_{1}+\theta_{2}-2}{n+\theta_{1}+\theta_{2}}}\notag\\
&\leq C\int_{B_{R}}|\nabla u|^{2}|x'|^{\theta_{3}}|x|^{\theta_{4}}dx.
\end{align*}
This, in combination with the H\"{o}lder's inequality, leads to that
\begin{align}\label{ZW001}
&\int_{B_{R}}|u|^{2\chi}|x'|^{\theta_{3}}|x|^{\theta_{4}}dx\notag\\
&=\int_{B_{R}}|u|^{2}|x'|^{\theta_{3}-\theta_{1}(\chi-1)}|x'|^{\theta_{4}-\theta_{2}(\chi-1)}|u|^{2(\chi-1)}|x'|^{\theta_{1}(\chi-1)}|x|^{\theta_{2}(\chi-1)}dx\notag\\
&\leq\left(\int_{B_{R}}|u|^{\frac{2}{2-\chi}}|x'|^{\frac{\theta_{3}-\theta_{1}(\chi-1)}{2-\chi}}|x|^{\frac{\theta_{4}-\theta_{2}(\chi-1)}{2-\chi}}dx\right)^{2-\chi}\left(\int_{B_{R}}|u|^{2}|x'|^{\theta_{1}}|x|^{\theta_{2}}dx\right)^{\chi-1}\notag\\
&\leq C\int_{B_{R}}|\nabla u|^{2}|x'|^{\theta_{3}}|x|^{\theta_{4}}dx\left(\int_{B_{R}}|u|^{2}|x'|^{\theta_{1}}|x|^{\theta_{2}}dx\right)^{\chi-1}.
\end{align}
Then integrating \eqref{ZW001} from $t_{1}$ to $t_{2}$, it follows from Young's inequality that
\begin{align*}
\left(\int^{t_{2}}_{t_{1}}\int_{B_{R}}|u|^{2\chi}w_{2}\right)^{\frac{1}{\chi}}\leq& C\left(\sup\limits_{t\in(t_{1},t_{2})}\int_{B_{R}}|u|^{2}w_{1}dx\right)^{\frac{\chi-1}{\chi}}\left(\int_{B_{R}}|\nabla u|^{2}w_{2}dx\right)^{\frac{1}{\chi}}\notag\\
\leq&C\left(\sup\limits_{t\in(t_{1},t_{2})}\int_{B_{R}}|u|^{2}w_{1}dx+\int^{t_{2}}_{t_{1}}\int_{B_{R}}|\nabla u|^{2}w_{2}dxdt\right).
\end{align*}
The proof is complete.

\end{proof}

The Caccioppoli inequality for the truncated solution is given as follows.
\begin{lemma}\label{lem003}
Set $\overline{m}\leq k\leq\overline{M}$. Then for any $B_{R}(x_{0})\subset B_{1}$ and nonnegative $\eta\in C^{\infty}(B_{R}(x_{0})\times(-1,0))$ which vanishes on $\partial B_{R}(x_{0})\times(-1,0)$, we obtain that for $-1<t_{1}<t_{2}<0$,
\begin{align*}
&\sup\limits_{t\in(t_{1},t_{2})}\int_{B_{R}(x_{0})}(v\eta)^{2}w_{1}dx+\int_{B_{R}(x_{0})\times(t_{1},t_{2})}|\nabla(v\eta)|^{2}w_{2}dxdt\notag\\
&\leq\int_{B_{R}(x_{0})}(v^{2}+C_{0}v^{3})\eta^{2}w_{1}dx\Big|_{t_{1}}+C_{0}\int_{B_{R}(x_{0})\times(t_{1},t_{2})}(\eta|\partial_{t}\eta|w_{1}+|\nabla\eta|^{2}w_{2})v^{2}dxdt,
\end{align*}
and
\begin{align*}
&\sup\limits_{t\in(t_{1},t_{2})}\int_{B_{R}(x_{0})}(\tilde{v}^{2}-C_{0}\tilde{v}^{3})\eta^{2}w_{1}dx+\int_{B_{R}(x_{0})}|\nabla(\tilde{v}\eta)|^{2}w_{2}dxdt\notag\\
&\leq\int_{B_{R}(x_{0})}(\tilde{v}\eta)^{2}w_{1}dx\Big|_{t_{1}}+C_{0}\int_{B_{R}(x_{0})\times(t_{1},t_{2})}(\eta|\partial_{t}\eta|w_{1}+|\nabla\eta|^{2}w_{2})\tilde{v}^{2}dxdt,
\end{align*}
where $C_{0}=C_{0}(n,p,\lambda,\overline{m},\overline{M})$, $v=(u-k)^{+}$, $\tilde{v}=(u-k)^{-}$, $u$ is the solution of problem \eqref{PROBLEM001}.

\end{lemma}

\begin{proof}
Choose test function $\varphi=v\eta^{2}$. By denseness, we obtain that for $t_{1}\leq s\leq t_{2}$,
\begin{align*}
\int^{s}_{t_{1}}\int_{B_{R}(x_{0})}w_{1}\partial_{t}u^{p}v\eta^{2}dxdt+\int^{s}_{t_{1}}\int_{B_{R}(x_{0})}Aw_{2}\nabla u\nabla(v\eta^{2})\,dxdt=0.
\end{align*}
On one hand,
\begin{align*}
\int^{s}_{t_{1}}\int_{B_{R}(x_{0})}w_{1}\partial_{t}u^{p}v\eta^{2}dxdt=&p \int^{s}_{t_{1}}\int_{B_{R}(x_{0})}w_{1}v\eta^{2}u^{p-1}\partial_{t}v\,dxdt\notag\\
=&p \int^{s}_{t_{1}}\int_{B_{R}(x_{0})}w_{1}\eta^{2}[(v+k)^{p}-k(v+k)^{p-1}]\partial_{t}v\,dxdt\notag\\
=&p \int^{s}_{t_{1}}\int_{B_{R}(x_{0})}w_{1}\eta^{2}\partial_{t}\mathcal{H},
\end{align*}
where
\begin{align*}
\mathcal{H}:=\frac{(v+k)^{p+1}}{p+1}-\frac{k(v+k)^{p}}{p}+\frac{k^{p+1}}{p(p+1)}.
\end{align*}
Remark that the last term $\frac{k^{p+1}}{p(p+1)}$ in $\mathcal{H}$ is added to keep its non-negative. In fact, by Taylor expansion, we obtain
\begin{align*}
\frac{(v+k)^{p+1}}{p+1}=&\frac{k^{p+1}}{p+1}\left(1+\frac{v}{k}\right)^{p+1}\notag\\
=&\frac{k^{p+1}}{p+1}\left(1+(p+1)\frac{v}{k}+\frac{p(p+1)}{2}\frac{v^{2}}{k^{2}}+\frac{(p-1)p(p+1)}{6}\frac{v^{3}}{k^{3}}+O\big(\frac{v^{4}}{k^{4}}\big)\right),
\end{align*}
and
\begin{align*}
\frac{k(v+k)^{p}}{p}=&\frac{k^{p+1}}{p}\left(1+\frac{v}{k}\right)^{p}\notag\\
=&\frac{k^{p+1}}{p}\left(1+p\frac{v}{k}+\frac{p(p-1)}{2}\frac{v^{2}}{k^{2}}+\frac{p(p-1)(p-2)}{6}\frac{v^{3}}{k^{3}}+O\big(\frac{v^{4}}{k^{4}}\big)\right).
\end{align*}
A consequence of these two relations shows that
\begin{align*}
\frac{(v+k)^{p+1}}{p+1}-\frac{k(v+k)^{p}}{p}=-\frac{k^{p+1}}{p(p+1)}+\frac{1}{2}k^{p-1}v^{2}+\frac{p-1}{3}k^{p-2}v^{3}+O\big(\frac{v^{4}}{k^{4}}\big),
\end{align*}
which yields that
\begin{align}\label{ZWZ005}
0\leq\mathcal{H}-\frac{1}{2}k^{p-1}v^{2}\leq C(p,\overline{m},\overline{M})v^{3}.
\end{align}
In light of \eqref{ZWZ005}, it follows from integration by parts that
\begin{align*}
&\int^{s}_{t_{1}}\int_{B_{R}(x_{0})}w_{1}\partial_{t}u^{p}v\eta^{2}dxdt\notag\\
&\geq\frac{p}{2}k^{p-1}\int_{B_{R}(x_{0})}w_{1}\eta^{2}(x,s)v^{2}(x,s)\,dx-\frac{p}{2}k^{p-1}\int_{B_{R}(x_{0})}w_{1}\eta^{2}(x,t_{1})v^{2}(x,t_{1})\,dx\notag\\
&\quad-C\int_{B_{R}(x_{0})}w_{1}\eta^{2}(x,t_{1})v^{3}(x,t_{1})\,dx-C\int^{s}_{t_{1}}\int_{B_{R}(x_{0})}w_{1}\eta|\partial_{t}\eta|v^{2}dxdt.
\end{align*}

On the other hand, utilizing Young's inequality, we have
\begin{align*}
&\int_{B_{R}(x_{0})\times(t_{1},s)}Aw_{2}\nabla u\nabla(v\eta^{2})\,dxdt\notag\\
&=\int_{B_{R}(x_{0})\times(t_{1},s)}Aw_{2}\nabla v(\eta^{2}\nabla v+2v\eta\nabla\eta)\,dxdt\notag\\
&\geq\frac{1}{\lambda}\int_{B_{R}(x_{0})\times(t_{1},s)}|\eta\nabla v|^{2}w_{2}dxdt-C\int_{B_{R}(x_{0})\times(t_{1},s)}v\eta|\nabla\eta||\nabla v|w_{2}dxdt\notag\\
&\geq\frac{1}{2\lambda}\int_{B_{R}(x_{0})\times(t_{1},s)}|\eta\nabla v|^{2}w_{2}dxdt-C\int_{B_{R}(x_{0})\times(t_{1},s)}v^{2}|\nabla\eta|^{2}w_{2}dxdt\notag\\
&\geq\frac{1}{4\lambda}\int_{B_{R}(x_{0})\times(t_{1},s)}|\nabla(\eta v)|^{2}w_{2}dxdt-C\int_{B_{R}(x_{0})\times(t_{1},s)}v^{2}|\nabla\eta|^{2}w_{2}dxdt,
\end{align*}
where in the last inequality we used the following elementary inequality:
\begin{align*}
|a-b|^{2}\geq\frac{1}{2}|a|^{2}-|b|^{2},\quad\text{for any $a,b\in\mathbb{R}^{n}$.}
\end{align*}
Therefore, the first inequality in Lemma \ref{lem003} holds.

The proof of the second inequality in Lemma \ref{lem003} is analogous by picking test function $\varphi=-\tilde{v}\eta^{2}$. Then we obtain
\begin{align*}
-\int^{s}_{t_{1}}\int_{B_{R}(x_{0})}w_{1}\partial_{t}u^{p}\tilde{v}\eta^{2}=&p \int^{s}_{t_{1}}\int_{B_{R}(x_{0})}w_{1}\tilde{v}\eta^{2}u^{p-1}\partial_{t}\tilde{v}\notag\\
=&p \int^{s}_{t_{1}}\int_{B_{R}(x_{0})}w_{1}\eta^{2}[-(k-\tilde{v})^{p}+k(k-\tilde{v})^{p-1}]\partial_{t}\tilde{v}\notag\\
=&p \int^{s}_{t_{1}}\int_{B_{R}(x_{0})}w_{1}\eta^{2}\partial_{t}\widetilde{\mathcal{H}},
\end{align*}
where
\begin{align*}
\widetilde{\mathcal{H}}=\frac{(k-\tilde{v})^{p+1}}{p+1}-\frac{k(k-\tilde{v})^{p}}{p}+\frac{k^{p+1}}{p(p+1)}.
\end{align*}
Similarly as before, it follows from Taylor expansion that
\begin{align*}
-C(p,\overline{m},\overline{M})\tilde{v}^{3}\leq\widetilde{\mathcal{H}}-\frac{1}{2}k^{p-1}\tilde{v}^{2}\leq0,
\end{align*}
which reads that
\begin{align*}
&\int^{s}_{t_{1}}\int_{B_{R}(x_{0})}w_{1}\partial_{t}u^{p}v\eta^{2}dxdt\notag\\
&\geq\frac{p}{2}k^{p-1}\int_{B_{R}(x_{0})}w_{1}\eta^{2}(x,s)\tilde{v}^{2}(x,s)\,dx-C\int_{B_{R}(x_{0})}w_{1}\eta^{2}(x,s)\tilde{v}^{3}(x,s)\,dx\notag\\
&\quad-\frac{p}{2}k^{q-1}\int_{B_{R}(x_{0})}w_{1}\eta^{2}(x,t_{1})\tilde{v}^{2}(x,t_{1})\,dx-C\int^{s}_{t_{1}}\int_{B_{R}(x_{0})}w_{1}\eta|\partial_{t}\eta|\tilde{v}^{2}dxdt.
\end{align*}
By the same argument as before, we have
\begin{align*}
&-\int_{B_{R}(x_{0})\times(t_{1},s)}Aw_{2}\nabla u\nabla(\tilde{v}\eta^{2})\,dxdt=\int_{B_{R}(x_{0})\times(t_{1},s)}Aw_{2}\nabla\tilde{v}\nabla(\tilde{v}\eta^{2})\,dxdt\notag\\
&\geq\frac{1}{4\lambda}\int_{B_{R}(x_{0})\times(t_{1},s)}|\nabla(\eta\tilde{v})|^{2}w_{2}dxdt-C\int_{B_{R}(x_{0})\times(t_{1},s)}\tilde{v}^{2}|\nabla\eta|^{2}w_{2}dxdt.
\end{align*}
The proof is complete.

\end{proof}

For $R>0$ and $(x_{0},t_{0})\in B_{1}\times[-1+R^{\theta_{1}+\theta_{2}},0]$, denote
\begin{align*}
Q_{R}(x_{0},t_{0}):=B_{R}(x_{0})\times(t_{0}-R^{\theta_{1}+\theta_{2}},t_{0}].
\end{align*}
For brevity, we use $Q_{R}$ to represent $Q_{R}(0,0)$ in the following. Introduce two Radon measures associated with the weights $w_{1}$ and $w_{2}$ as follows:
\begin{align*}
d\mu_{w_{i}}=w_{i}dx,\quad d\nu_{w_{i}}=w_{i}dxdt,\quad i=1,2,
\end{align*}
satisfying that for $E\subset B_{1}$ and $\widetilde{E}\subset Q_{1}$,
\begin{align*}
|E|_{\mu_{w_{i}}}=\int_{E}w_{i}dx,\quad |\widetilde{E}|_{\nu_{w_{i}}}=\int_{\widetilde{E}}w_{i}dxdt.
\end{align*}
Observe that by H\"{o}lder's inequality, we know that for $\widetilde{E}\subset Q_{R}$,
\begin{align}\label{QAZ002}
\frac{|\widetilde{E}|_{\nu_{w_{2}}}}{|Q_{R}|_{\nu_{w_{2}}}}\leq&\leq\frac{C|\widetilde{E}|_{\nu_{w_{1}}}^{\frac{\theta_{3}}{\theta_{1}}}R^{\frac{(n+\theta_{1}+\theta_{2})(\theta_{1}-\theta_{3})}{\theta_{1}}}}{R^{n+\theta_{1}+\theta_{2}+\theta_{3}+\theta_{4}}}\leq C\left(\frac{|\widetilde{E}|_{\nu_{w_{1}}}}{|Q_{R}|_{\nu_{w_{1}}}}\right)^{\frac{\theta_{3}}{\theta_{1}}},
\end{align}
where $C=C(n,\theta_{1},\theta_{2},\theta_{3},\theta_{4}).$ Here we used the assumed condition that $\theta_{1}/\theta_{3}=\theta_{2}/\theta_{4}$, $\theta_{3},\theta_{4}\neq0.$

Similar to Lemma \ref{LEM0035ZZW}, we improve the oscillation of the solution $u$ in a small region as follows.
\begin{lemma}\label{LEM0035}
Assume as in Theorem \ref{ZWTHM90} or Theorem \ref{THM060}. For $R\in(0,1)$ and $t_{0}\in[-1+R^{\theta_{1}+\theta_{2}},0]$, let $0<\overline{m}\leq m\leq\inf\limits_{Q_{R}(0,t_{0})}u\leq\sup\limits_{Q_{R}(0,t_{0})}u\leq M\leq\overline{M}$. Then

$(a)$ there exists a small constant $0<\tau_{0}=\tau_{0}(n,p,\theta_{1},\theta_{2},\theta_{3},\lambda,\overline{m},\overline{M})<1$ such that for any $\varepsilon>0$ and $0<\tau<\tau_{0}$, if
\begin{align}\label{ZWZ007}
|\{(x,t)\in Q_{R}(0,t_{0}): u(x,t)>M-\varepsilon\}|_{\nu_{w_{1}}}\leq\tau|Q_{R}(0,t_{0})|_{\nu_{w_{1}}},
\end{align}
then we have
\begin{align}\label{DZ001}
u(x,t)\leq M-\frac{\varepsilon}{2},\quad\text{for $(x,t)\in Q_{R/2}(0,t_{0})$;}
\end{align}

$(b)$ there exists two small constant $0<\varepsilon_{0}=\varepsilon_{0}(n,p,\lambda,\overline{m},\overline{M})<1$ and $0<\tau_{0}=\tau_{0}(n,p,\theta_{1},\theta_{2},\theta_{3},\lambda,\overline{m},\overline{M})<1$ such that for any $0<\varepsilon\leq\varepsilon_{0}$ and $0<\tau<\tau_{0}$, if
\begin{align*}
|\{(x,t)\in Q_{R}(0,t_{0}): u(x,t)<m+\varepsilon\}|_{\nu_{w_{1}}}\leq\tau|Q_{R}(0,t_{0})|_{\nu_{w_{1}}},
\end{align*}
then we have
\begin{align}\label{ZWZ009}
u(x,t)\geq m-\frac{\varepsilon}{2},\quad\text{for $(x,t)\in Q_{R/2}(0,t_{0})$.}
\end{align}

\end{lemma}

\begin{remark}
From the proof of Lemma \ref{LEM0035} below, we see that the value of $\theta_{3}+\theta_{4}$ has to be restricted to $2$ and thus affect our final regular index in Theorems \ref{ZWTHM90} and \ref{THM060}.

\end{remark}

\begin{proof}
Without loss of generality, let $t_{0}=0$.

\noindent{\bf Step 1.}
For $\varepsilon>0$ and $i=0,1,2,...,$ set
\begin{align*}
r_{i}=\frac{R}{2}+\frac{R}{2^{i+1}},\quad k_{i}=M-\varepsilon+\frac{\varepsilon}{2}(1-2^{-i}).
\end{align*}
Choose a cut-off function $\eta_{i}\in C_{0}^{\infty}(Q_{r_{i}})$ such that
\begin{align*}
\eta_{i}=1\;\mathrm{in}\;Q_{r_{i+1}},\quad 0\leq\eta_{i}\leq1,\;|\nabla\eta_{i}|\leq\frac{C}{r_{i}-r_{i+1}},\;|\partial_{t}\eta_{i}|\leq\frac{C}{r_{i}^{\theta_{1}+\theta_{2}}-r_{i+1}^{\theta_{1}+\theta_{2}}}\;\,\mathrm{in}\;Q_{r_{i}}.
\end{align*}
Denote $v_{i}=(u-k_{i})^{+}$ and $A(k,\rho)=\{(x,t)\in Q_{\rho}: u>k\}$ for $k\in[m,M]$ and $\rho\in(0,R]$. Then combining Proposition \ref{prop001} and Lemma \ref{lem003}, we deduce
\begin{align}\label{AZQ001}
\left(\int_{Q_{R}}|\eta_{i}v_{i}|^{2\chi}w_{2}\right)^{\frac{1}{\chi}}\leq C\int_{Q_{R}}\left(|\nabla\eta_{i}|^{2}+|\partial_{t}\eta_{i}||x'|^{\theta_{1}-\theta_{3}}|x|^{\theta_{2}-\theta_{4}}\right)v_{i}^{2}w_{2},
\end{align}
where $\chi=(n+\theta_{1}+\theta_{2}+2)(n+\theta_{1}+\theta_{2})^{-1}$. Note that
\begin{align}\label{ZMQT001}
|\partial_{t}\eta_{i}|R^{\theta_{1}+\theta_{2}-\theta_{3}-\theta_{4}}\leq\frac{CR^{2-\theta_{3}-\theta_{4}}}{(r_{i}-r_{i+1})^{2}}=\frac{C}{(r_{i}-r_{i+1})^{2}}.
\end{align}
Therefore, we have
\begin{align*}
\int_{Q_{R}}\left(|\nabla\eta_{i}|^{2}+|\partial_{t}\eta_{i}||x'|^{\theta_{1}-\theta_{3}}|x|^{\theta_{2}-\theta_{4}}\right)v_{i}^{2}w_{2}\leq\frac{C(M-k_{i})^{2}}{(r_{i}-r_{i+1})^{2}}|A(k_{i},r_{i})|_{\nu_{w_{2}}},
\end{align*}
and
\begin{align*}
\left(\int_{Q_{R}}|\eta_{i}v_{i}|^{2\chi}w_{2}dxdt\right)^{\frac{1}{\chi}}\geq(k_{i+1}-k_{i})^{2}|A(k_{i+1},r_{i+1})|_{\nu_{w_{2}}}^{\frac{1}{\chi}}.
\end{align*}
Define
\begin{align*}
F_{i}:=\frac{|A(k_{i},r_{i})|_{\nu_{w_{2}}}}{|Q_{R}|_{\nu_{w_{2}}}}.
\end{align*}
Then we have
\begin{align}\label{QAZ001}
F_{i+1}\leq&(C2^{4(i+2)})^{\chi}R^{(n+\theta_{1}+\theta_{2}+\theta_{3}+\theta_{4})(\chi-1)-2\chi}F_{i}^{\chi}\notag\\
=&(C2^{4(i+2)})^{\chi}R^{\frac{2(\theta_{3}+\theta_{4}-2)}{n+\theta_{1}+\theta_{2}}}F_{i}^{\chi}=(C2^{4(i+2)})^{\chi}F_{i}^{\chi}\notag\\
\leq&\prod\limits^{i}_{s=0}(C2^{4(i+2-s)})^{\chi^{s+1}}F_{0}^{\chi^{i+1}}.
\end{align}
Observe from \eqref{ZMQT001}--\eqref{QAZ001} that the value of $\theta_{3}+\theta_{4}$ must be chosen to be $2$. A consequence of \eqref{QAZ002} and \eqref{QAZ001} shows that there exists a constant $i_{0}=i_{0}(n,\theta_{1},\theta_{2},\theta_{3})>0$ such that if $i\geq i_{0}$,
\begin{align*}
F_{i+1}\leq (C^{\ast})^{i+1}\left(\frac{|A(k_{0},r_{0})|_{\nu_{w_{1}}}}{|Q_{R}|_{\nu_{w_{1}}}}\right)^{\frac{\theta_{3}}{\theta_{1}}\chi^{i+1}}\leq (C^{\ast})^{i+1}\left(\frac{|A(k_{0},r_{0})|_{\nu_{w_{1}}}}{|Q_{R}|_{\nu_{w_{1}}}}\right)^{\frac{\theta_{3}}{\theta_{1}}\chi(i+1)}.
\end{align*}
By taking $\tau_{0}=(C^{\ast})^{-\frac{\theta_{1}}{\theta_{3}}\chi}$, we obtain that for any $\varepsilon>0$ and $0<\tau<\tau_{0}$, if \eqref{ZWZ007} holds, then
\begin{align*}
F_{i+1}\leq \Big(C^{\ast}\tau^{\frac{\theta_{3}}{\theta_{1}}\chi}\Big)^{i+1}=\left(\frac{\tau}{\tau_{0}}\right)^{\frac{\theta_{3}}{\theta_{1}}\chi(i+1)}\rightarrow0,\quad\mathrm{as}\;i\rightarrow\infty.
\end{align*}
That is, \eqref{DZ001} holds.

\noindent{\bf Step 2.} Analogously as before, pick
\begin{align*}
r_{i}=\frac{R}{2}+\frac{R}{2^{i+1}},\quad \tilde{k}_{i}=m+\varepsilon-\frac{\varepsilon}{2}(1-2^{-i}),\quad i\geq0.
\end{align*}
Let $\varepsilon_{0}=\frac{1}{C_{0}}$, where $C_{0}$ is given in Lemma \ref{lem003}. Denote $\tilde{v}_{i}=(u-\tilde{k}_{i})^{-}$. Then we obtain that for any $0<\varepsilon\leq\varepsilon_{0}$,
\begin{align*}
\tilde{v}_{i}^{2}-C_{0}\tilde{v}_{i}^{3}\geq(1-C_{0}\varepsilon)\tilde{v}_{i}^{2}\geq0,
\end{align*}
which implies that \eqref{AZQ001} holds with $v_{i}$ replaced by $\tilde{v}_{i}$. Then following the left proof of \eqref{DZ001} above, we deduce that \eqref{ZWZ009} holds. The proof is complete.

\end{proof}

The decaying estimates for the distribution function of $u$ are stated as follows.
\begin{lemma}\label{lem005}
Let the values of $n,p,q,\theta_{i},i=1,2,3,4$ be assumed in Theorem \ref{ZWTHM90} or Theorem \ref{THM060} with $\theta_{3}+\theta_{4}=2$ replaced by $0\leq\theta_{3}+\theta_{4}\leq2$. Suppose that $0<\gamma<1$, $0<R<\frac{1}{2}$, $0<a\leq1$, $-\frac{1}{2}<t_{0}\leq-aR^{\theta_{1}+\theta_{2}}$, and $\overline{m}\leq m_{a}\leq\inf\limits_{B_{2R}\times[t_{0},t_{0}+a R^{\theta_{1}+\theta_{2}}]}u\leq\sup\limits_{B_{2R}\times[t_{0},t_{0}+a R^{\theta_{1}+\theta_{2}}]}u\leq M_{a}\leq\overline{M}$. Then

$(a)$ for any $\varepsilon>0$, if
\begin{align*}
\frac{|\{x\in B_{R}:u(x,t)>M_{a}-\varepsilon\}|_{\mu_{w_{1}}}}{|B_{R}|_{\mu_{w_{1}}}}\leq1-\gamma,\quad\forall\;t\in[t_{0},t_{0}+aR^{\theta_{1}+\theta_{2}}],
\end{align*}
then for any $j\geq1$,
\begin{align}\label{DEC001}
\frac{|\{(x,t)\in B_{R}\times[t_{0},t_{0}+aR^{\theta_{1}+\theta_{2}}]:u(x,t)>M_{a}-\frac{\varepsilon}{2^{j}}\}|_{\nu_{w_{1}}}}{|B_{R}\times[t_{0},t_{0}+aR^{\theta_{1}+\theta_{2}}]|_{\nu_{w_{1}}}}\leq\frac{C}{\sqrt[q]{\gamma}\sqrt{a}j^{\frac{2-q}{2q}}};
\end{align}

$(b)$ for any $0<\varepsilon\leq\varepsilon_{0}=C_{0}^{-1}$ with $C_{0}=C_{0}(n,p,\lambda,\overline{m},\overline{M})$ given by Lemma \ref{lem003}, if
\begin{align*}
\frac{|\{x\in B_{R}:u(x,t)<m_{a}+\varepsilon\}|_{\mu_{w_{1}}}}{|B_{R}|_{\mu_{w_{1}}}}\leq1-\gamma,\quad\forall\;t\in[t_{0},t_{0}+aR^{\theta_{1}+\theta_{2}}],
\end{align*}
then for any $j\geq1$,
\begin{align*}
\frac{|\{(x,t)\in B_{R}\times[t_{0},t_{0}+aR^{\theta_{1}+\theta_{2}}]:u(x,t)<m_{a}+\frac{\varepsilon}{2^{j}}\}|_{\nu_{w_{1}}}}{|B_{R}\times[t_{0},t_{0}+aR^{\theta_{1}+\theta_{2}}]|_{\nu_{w_{1}}}}\leq\frac{C}{\sqrt[q]{\gamma}\sqrt{a}j^{\frac{2-q}{2q}}},
\end{align*}
where $C=C(n,p,q,\theta_{1},\theta_{2},\theta_{3},\lambda,\overline{m},\overline{M})$.

\end{lemma}

\begin{remark}
Since the proof of Lemma \ref{lem005} only uses the aforementioned Proposition \ref{prop002} and Lemma \ref{lem003} instead of Lemma \ref{LEM0035}, we can obtain a larger range for the value of $\theta_{3}+\theta_{4}$ than that in Lemma \ref{LEM0035}.

\end{remark}

\begin{proof}
\noindent{\bf Step 1.}
For $i\geq0$, denote $k_{i}=M_{a}-\frac{\varepsilon}{2^{i}}$ and
\begin{align*}
A(k_{i},R;t)=B_{R}\cap\{u(\cdot,t)>k_{i}\},\quad A(k_{i},R)=(B_{R}\times[t_{0},t_{0}+aR^{\theta_{1}+\theta_{2}}])\cap\{u>k_{i}\}.
\end{align*}
It then follows from \eqref{pro001} that for $1<q<2,$
\begin{align}\label{WMZ001}
&(k_{i+1}-k_{i})^{q}|A(k_{i+1},R;t)|_{\mu_{w_{1}}}^{q}|B_{R}\setminus A(k_{i},R;t)|_{\mu_{w_{1}}}\notag\\
&\leq CR^{q(n+\theta_{1}+\theta_{2}+1)}\int_{A(k_{i},R;t)\setminus A(k_{i+1},R;t)}|\nabla u|^{q}w_{1}dx.
\end{align}
From the assumed condition, we have
\begin{align*}
|B_{R}\setminus A(k_{i},R;t)|_{\mu_{w_{1}}}\geq\gamma|B_{R}|_{\mu_{w_{1}}}=C(n,\theta_{1})\gamma R^{n+\theta_{1}+\theta_{2}}.
\end{align*}
Substituting this into \eqref{WMZ001} and integrating from $t_{0}$ to $t_{0}+a R^{\theta_{1}+\theta_{2}}$, we deduce from H\"{o}lder's inequality that
\begin{align*}
&\int^{t_{0}+aR^{\theta_{1}+\theta_{2}}}_{t_{0}}|A(k_{i+1},R;t)|_{\mu_{w_{1}}}dt\notag\\
&\leq\frac{C2^{i+1}}{\varepsilon\sqrt[q]{\gamma}}R^{\frac{(n+\theta_{1}+\theta_{2})(q-1)}{q}+1}\int_{t_{0}}^{t_{0}+a R^{\theta_{1}+\theta_{2}}}\left(\int_{A(k_{i},R;t)\setminus A(k_{i+1},R;t)}|\nabla u|^{q}w_{1}dx\right)^{\frac{1}{q}}dt\notag\\
&\leq\frac{C2^{i+1}a^{\frac{q-1}{q}}}{\varepsilon\sqrt[q]{\gamma}}R^{\frac{(n+2\theta_{1}+2\theta_{2})(q-1)}{q}+1}\left(\int_{A(k_{i},R)\setminus A(k_{i+1},R)}|\nabla u|^{q}w_{1}dxdt\right)^{\frac{1}{q}}.
\end{align*}
In light of $1<q<2$, it follows from H\"{o}lder's inequality again that
\begin{align*}
&\left(\int_{A(k_{i},R)\setminus A(k_{i+1},R)}|\nabla u|^{q}w_{1}\right)^{\frac{1}{q}}\notag\\
&\leq\left(\int_{A(k_{i},R)\setminus A(k_{i+1},R)}|\nabla u|^{2}w_{2}\right)^{\frac{1}{2}}\left(\int_{A(k_{i},R)\setminus A(k_{i+1},R)}|x'|^{\frac{2\theta_{1}-q\theta_{3}}{2-q}}|x|^{\frac{2\theta_{2}-q\theta_{4}}{2-q}}\right)^{\frac{2-q}{2q}}\notag\\
&\leq R^{\frac{\theta_{1}+\theta_{2}-\theta_{3}-\theta_{4}}{2}}|A(k_{i},R)\setminus A(k_{i+1},R)|_{\nu_{w_{1}}}^{\frac{2-q}{2q}}\left(\int_{B_{R}\times[t_{0},t_{0}+aR^{\theta_{1}+\theta_{2}}]}|\nabla (u-k_{i})^{+}|^{2}w_{2}\right)^{\frac{1}{2}}.
\end{align*}
Pick a cut-off function $\eta\in C_{0}^{\infty}(B_{2R})$ such that
\begin{align}\label{ETA003}
\eta=1\;\mathrm{in}\;B_{R},\quad0\leq\eta\leq1,\;|\nabla\eta|\leq\frac{C(n)}{R}\;\,\mathrm{in}\;B_{2R}.
\end{align}
Then from Lemma \ref{lem003}, we deduce
\begin{align*}
&\left(\int_{B_{R}\times[t_{0},t_{0}+aR^{\theta_{1}+\theta_{2}}]}|\nabla (u-k_{i})^{+}|^{2}w_{2}\right)^{\frac{1}{2}}\notag\\
&\leq C\left(\int_{B_{2R}}|(u-k_{i})^{+}(x,t_{0})|^{2}w_{1}+\frac{1}{R^{2}}\int_{B_{2R}\times[t_{0},t_{0}+aR^{\theta_{1}+\theta_{2}}]}|(u-k_{i})^{+}|^{2}w_{2}\right)^{\frac{1}{2}}\notag\\
&\leq\frac{C\varepsilon}{2^{i}}R^{\frac{n+\theta_{1}+\theta_{2}+\theta_{3}+\theta_{4}-2}{2}},
\end{align*}
where in the last inequality we used the assumed condition that $0\leq\theta_{3}+\theta_{4}\leq2$. Therefore, combining these above facts, we obtain
\begin{align*}
|A(k_{i+1},R)|_{\nu_{w_{1}}}\leq&\frac{Ca^{\frac{q-1}{q}}}{\sqrt[q]{\gamma}}R^{\frac{(n+2\theta_{1}+2\theta_{2})(3q-2)}{2q}}|A(k_{i},R)\setminus A(k_{i+1},R)|_{\nu_{w_{1}}}^{\frac{2-q}{2q}},
\end{align*}
which yields that for $j\geq1$,
\begin{align*}
j|A(k_{j},R)|_{\nu_{w_{1}}}^{\frac{2q}{2-q}}\leq&\sum^{j-1}_{i=0}|A(k_{i+1},R)|_{\nu_{w_{1}}}^{\frac{2q}{2-q}}\notag\\
\leq&\frac{Ca^{\frac{2(q-1)}{2-q}}}{\gamma^{\frac{2}{2-q}}}R^{\frac{(n+2\theta_{1}+2\theta_{2})(3q-2)}{2-q}}|B_{R}\times[t_{0},t_{0}+aR^{\theta_{1}+\theta_{2}}]|_{\nu_{w_{1}}}\notag\\
\leq&\frac{C}{\gamma^{\frac{2}{2-q}}a^{\frac{q}{2-q}}}|B_{R}\times[t_{0},t_{0}+aR^{\theta_{1}+\theta_{2}}]|_{\nu_{w_{1}}}^{\frac{2q}{2-q}}.
\end{align*}
Then \eqref{DEC001} holds.

\noindent{\bf Step 2.}
For $i\geq0$, write
\begin{align*}
\tilde{k}_{i}=m_{a}+\frac{\varepsilon}{2^{i}},
\end{align*}
and
\begin{align*}
\tilde{A}(k_{i},R;t)=B_{R}\cap\{u(\cdot,t)<k_{i}\},\quad \tilde{A}(k_{i},R)=(B_{R}\times[t_{0},t_{0}+aR^{\theta_{1}+\theta_{2}}])\cap\{u<k_{i}\}.
\end{align*}
Using \eqref{pro002}, we obtain that for $1<q<2,$
\begin{align*}
&(\tilde{k}_{i}-\tilde{k}_{i+1})^{q}|\tilde{A}(\tilde{k}_{i+1},R;t)|_{\mu_{w_{1}}}^{q}|B_{R}\setminus \tilde{A}(\tilde{k}_{i},R;t)|_{\mu_{w_{1}}}\notag\\
&\leq CR^{q(n+\theta_{1}+\theta_{2}+1)}\int_{\tilde{A}(\tilde{k}_{i},R;t)\setminus \tilde{A}(\tilde{k}_{i+1},R;t)}|\nabla u|^{q}w_{1}dx.
\end{align*}
Observe by the assumed condition that
\begin{align*}
|B_{R}\setminus \tilde{A}(\tilde{k}_{i},R;t)|_{\mu_{w_{1}}}\geq\gamma|B_{R}|_{\mu_{w_{1}}}=C(n,\theta_{1},\theta_{2})\gamma R^{n+\theta_{1}+\theta_{2}}.
\end{align*}
Analogously as above, integrating from $t_{0}$ to $t_{0}+a R^{\theta_{1}+\theta_{2}}$ and using H\"{o}lder's inequality, we have
\begin{align*}
&\int^{t_{0}+aR^{\theta_{1}+\theta_{2}}}_{t_{0}}|\tilde{A}(\tilde{k}_{i+1},R;t)|_{\mu_{w_{1}}}dt\notag\\
&\leq\frac{C2^{i+1}a^{\frac{q-1}{q}}}{\varepsilon\sqrt[q]{\gamma}}R^{\frac{(n+2\theta_{1}+2\theta_{2})(q-1)}{q}+1}\left(\int_{\tilde{A}(\tilde{k}_{i},R)\setminus \tilde{A}(\tilde{k}_{i+1},R)}|\nabla u|^{q}w_{1}dxdt\right)^{\frac{1}{q}},
\end{align*}
and
\begin{align*}
&\int_{\tilde{A}(\tilde{k}_{i},R)\setminus A(\tilde{k}_{i+1},R)}|\nabla u|^{q}w_{1}\notag\\
&\leq\left(\int_{\tilde{A}(\tilde{k}_{i},R)\setminus \tilde{A}(\tilde{k}_{i+1},R)}|\nabla u|^{2}w_{2}\right)^{\frac{q}{2}}\left(\int_{\tilde{A}(\tilde{k}_{i},R)\setminus \tilde{A}(\tilde{k}_{i+1},R)}|x'|^{\frac{2\theta_{1}-q\theta_{3}}{2-q}}|x|^{\frac{2\theta_{2}-q\theta_{4}}{2-q}}\right)^{\frac{2-q}{2}}\notag\\
&\leq R^{\frac{q(\theta_{1}+\theta_{2}-\theta_{3}-\theta_{4})}{2}}|\tilde{A}(\tilde{k}_{i},R)\setminus \tilde{A}(\tilde{k}_{i+1},R)|_{\nu_{w_{1}}}^{\frac{2-q}{2q}}\left(\int_{B_{R}\times[t_{0},t_{0}+aR^{\theta_{1}+\theta_{2}}]}|\nabla (u-\tilde{k}_{i})^{-}|^{2}w_{2}\right)^{\frac{q}{2}}.
\end{align*}
For any $0<\varepsilon\leq\varepsilon_{0}=C_{0}^{-1}$ with $C_{0}$ given in Lemma \ref{lem003}, we know that
\begin{align*}
[(u-\tilde{k}_{i})^{-}]^{2}-C_{0}[(u-\tilde{k}_{i})^{-}]^{3}\geq(1-C_{0}\varepsilon)[(u-\tilde{k}_{i})^{-}]^{2}\geq0.
\end{align*}
Therefore, in view of $0\leq\theta_{3}+\theta_{4}\leq2$ and applying Lemma \ref{lem003} with $\eta$ defined by \eqref{ETA003}, we derive
\begin{align*}
&\int_{B_{R}\times[t_{0},t_{0}+aR^{\theta_{1}+\theta_{2}}]}|\nabla (u-\tilde{k}_{i})^{-}|^{2}w_{2}\notag\\
&\leq C\left(\int_{B_{2R}}|(u-\tilde{k}_{i})^{-}(x,t_{0})|^{2}w_{1}+\frac{1}{R^{2}}\int_{B_{R}\times[t_{0},t_{0}+aR^{\theta_{1}+\theta_{2}}]}|(u-\tilde{k}_{i})^{-}|^{2}w_{2}\right)\notag\\
&\leq\frac{C\varepsilon^{2}}{4^{i}}R^{n+\theta_{1}+\theta_{2}+\theta_{3}+\theta_{4}-2}.
\end{align*}
Then we deduce
\begin{align*}
|\tilde{A}(\tilde{k}_{i+1},R)|_{\nu_{w_{1}}}^{\frac{2q}{2-q}}\leq&\frac{Ca^{\frac{2(q-1)}{2-q}}}{\gamma^{\frac{2}{2-q}}}R^{\frac{(n+2\theta_{1}+2\theta_{2})(3q-2)}{2-q}}|\tilde{A}(\tilde{k}_{i},R)\setminus \tilde{A}(\tilde{k}_{i+1},R)|_{\nu_{w_{1}}}.
\end{align*}
This leads to that for $j\geq1$,
\begin{align*}
j|\tilde{A}(\tilde{k}_{j},R)|_{\nu_{w_{1}}}^{\frac{2q}{2-q}}\leq&\sum^{j-1}_{i=0}|\tilde{A}(\tilde{k}_{i+1},R)|_{\nu_{w_{1}}}^{\frac{2q}{2-q}}\notag\\
\leq&\frac{Ca^{\frac{2(q-1)}{2-q}}}{\gamma^{\frac{2}{2-q}}}R^{\frac{(n+2\theta_{1}+2\theta_{2})(3q-2)}{2-q}}|B_{R}\times[t_{0},t_{0}+aR^{\theta_{1}+\theta_{2}}]|_{\nu_{w_{1}}}\notag\\
\leq&\frac{C}{\gamma^{\frac{2}{2-q}}a^{\frac{q}{2-q}}}|B_{R}\times[t_{0},t_{0}+aR^{\theta_{1}+\theta_{2}}]|_{\nu_{w_{1}}}^{\frac{2q}{2-q}}.
\end{align*}
The proof is complete.

\end{proof}

We now give explicit estimates for the distribution function of $u$ at each time slice from the starting time.
\begin{lemma}\label{LEM006}
Assume as in Theorem \ref{ZWTHM90} or Theorem \ref{THM060}. Let $0<\gamma<1$, $0<R<\frac{1}{2}$, $-\frac{1}{2}<t_{0}\leq-R^{\theta_{1}+\theta_{2}}$ and $\overline{m}\leq m_{1}\leq\inf\limits_{B_{2R}\times[t_{0},t_{0}+R^{\theta_{1}+\theta_{2}}]}u\leq\sup\limits_{B_{2R}\times[t_{0},t_{0}+R^{\theta_{1}+\theta_{2}}]}u\leq M_{1}\leq\overline{M}$. Then there exist a small constant $\bar{\varepsilon}_{0}=\bar{\varepsilon}_{0}(n,p,\theta_{1},\theta_{2},\lambda,\gamma,\overline{m},\overline{M})>0$ and a large constant $\bar{l}_{0}=\bar{l}_{0}(n,p,q,\theta_{1},\theta_{2},\theta_{3},\lambda,\gamma,\overline{m},\overline{M})>1$ such that

$(i)$ for every $0<\varepsilon\leq\bar{\varepsilon}_{0}$, if
\begin{align}\label{WMDN001}
\frac{|\{x\in B_{R}: u(x,t_{0})>M_{1}-\varepsilon\}|_{\mu_{w_{1}}}}{|B_{R}|_{\mu_{w_{1}}}}\leq1-\gamma,
\end{align}
then for any $t_{0}\leq t\leq t_{0}+R^{\theta_{1}+\theta_{2}},$
\begin{align}\label{DQAF001}
\frac{|\{x\in B_{R}: u(x,t)>M_{1}-2^{-l_{0}}\varepsilon\}|_{\mu_{w_{1}}}}{|B_{R}|_{\mu_{w_{1}}}}\leq1-\frac{\gamma}{2};
\end{align}

$(ii)$ for every $0<\varepsilon\leq\bar{\varepsilon}_{0}$, if
\begin{align}\label{QPZA001}
\frac{|\{x\in B_{R}: u(x,t_{0})<m_{1}+\varepsilon\}|_{\mu_{w_{1}}}}{|B_{R}|_{\mu_{w_{1}}}}\leq1-\gamma,
\end{align}
then for any $t_{0}\leq t\leq t_{0}+R^{\theta_{1}+\theta_{2}},$
\begin{align}\label{WZPMQ001}
\frac{|\{x\in B_{R}: u(x,t)<m_{1}+2^{-l_{0}}\varepsilon\}|_{\mu_{w_{1}}}}{|B_{R}|_{\mu_{w_{1}}}}\leq1-\frac{\gamma}{2}.
\end{align}

\end{lemma}

\begin{remark}
It is worth emphasizing that the explicit values of $\bar{\varepsilon}_{0}$ and $\bar{l}_{0}$ are given by \eqref{MWQK001} and \eqref{MWQK002} below.

\end{remark}

\begin{proof}
\noindent{\bf Step 1.}
For $a\in(0,1]$ and $k\in[\overline{m},\overline{M}]$, define
\begin{align*}
A^{a}(k,R)=(B_{R}\times[t_{0},t_{0}+a R^{\theta_{1}+\theta_{2}}])\cap\{u>k\}.
\end{align*}
Take a smooth cut-off function $\eta\in C^{\infty}_{0}(B_{R})$ satisfying that $\eta=1$ in $B_{\sigma R}$, where $\sigma\in(0,1)$ to be determined later. Set $k_{1}>1$. Denote $v=(u-(M_{1}-\varepsilon))^{+}$. From Lemma \ref{lem003}, we obtain
\begin{align}\label{QNE089}
&\sup\limits_{t\in(t_{0},t_{0}+aR^{\theta_{1}+\theta_{2}})}\int_{B_{R}}v^{2}\eta^{2}w_{1}dx\notag\\
&\leq\int_{B_{R}}(v^{2}+Cv^{3})\eta^{2}w_{1}dx\big|_{t_{0}}+C\int_{B_{R}\times[t_{0},t_{0}+aR^{\theta_{1}+\theta_{2}}]}v^{2}|\nabla\eta|^{2}w_{2}dxdt.
\end{align}
Observe that for $t\in[t_{0},t_{0}+aR^{\theta_{1}+\theta_{2}}]$,
\begin{align*}
\int_{B_{R}}v^{2}\eta^{2}w_{1}dx\big|_{t}\geq\varepsilon^{2}(1-2^{-k_{1}})^{2}|B_{\sigma R}\cap\{u(x,t)>M_{1}-2^{-k_{1}}\varepsilon\}|_{\mu_{w_{1}}},
\end{align*}
and by \eqref{WMDN001},
\begin{align*}
\int_{B_{R}}(v^{2}+Cv^{3})\eta^{2}w_{1}dx\big|_{t_{0}}\leq&\varepsilon^{2}(1+C\varepsilon)|\{x\in B_{R}: u(x,t_{0})>M_{1}-\varepsilon\}|_{\mu_{w_{1}}}\notag\\
\leq&\varepsilon^{2}(1+C\varepsilon)(1-\gamma)|B_{R}|_{\mu_{w_{1}}},
\end{align*}
and
\begin{align*}
\int^{t_{0}+aR^{\theta_{1}+\theta_{2}}}_{t_{0}}\int_{B_{R}}v^{2}|\nabla\eta|^{2}w_{2}dxdt\leq&\frac{C\varepsilon^{2}}{(1-\sigma)^{2}R^{2}}|A^{a}(M_{1}-\varepsilon,R)|_{\nu_{w_{2}}}\notag\\
\leq&\frac{C\varepsilon^{2}R^{\theta_{3}+\theta_{4}-2}}{(1-\sigma)^{2}}|B_{R}|_{\mu_{w_{1}}}\frac{|A^{a}(M_{1}-\varepsilon,R)|_{\nu_{w_{2}}}}{|Q_{R}|_{\nu_{w_{2}}}}\notag\\
=&\frac{C\varepsilon^{2}}{(1-\sigma)^{2}}|B_{R}|_{\mu_{w_{1}}}\frac{|A^{a}(M_{1}-\varepsilon,R)|_{\nu_{w_{2}}}}{|Q_{R}|_{\nu_{w_{2}}}},
\end{align*}
where we utilized the assumed condition of $\theta_{3}+\theta_{4}=2$. A consequence of these facts gives that for $t\in[t_{0},t_{0}+aR^{\theta_{1}+\theta_{2}}]$,
\begin{align*}
&|B_{\sigma R}\cap\{u(x,t)>M_{1}-2^{-k_{1}}\varepsilon\}|_{\mu_{w_{1}}}\notag\\
&\leq|B_{R}|_{\mu_{w_{1}}}\left(\frac{(1+C\varepsilon)(1-\gamma)}{(1-2^{-k_{1}})^{2}}+\frac{C}{(1-\sigma)^{2}}\frac{|A^{a}(M_{1}-\varepsilon,R)|_{\nu_{w_{2}}}}{|Q_{R}|_{\nu_{w_{2}}}}\right),
\end{align*}
which, together with the fact that $|B_{R}\setminus B_{\sigma R}|_{\mu_{w_{1}}}\leq C(1-\sigma)|B_{R}|_{\mu_{w_{1}}}$, reads that
\begin{align*}
&\frac{|B_{R}\cap\{u(x,t)>M_{1}-2^{-k_{1}}\varepsilon\}|_{\mu_{w_{1}}}}{|B_{R}|_{\mu_{w_{1}}}}\notag\\
&\leq \frac{(1+C\varepsilon)(1-\gamma)}{(1-2^{-k_{1}})^{2}}+\frac{C}{(1-\sigma)^{2}}\bigg((1-\sigma)^{3}+\frac{|A^{a}(M_{1}-\varepsilon,R)|_{\nu_{w_{2}}}}{|Q_{R}|_{\nu_{w_{2}}}}\bigg).
\end{align*}
Pick $\sigma$ such that
\begin{align*}
(1-\sigma)^{3}=\frac{|A^{a}(M_{1}-\varepsilon,R)|_{\nu_{w_{2}}}}{|Q_{R}|_{\nu_{w_{2}}}},
\end{align*}
which yields that for $t\in[t_{0},t_{0}+aR^{\theta_{1}+\theta_{2}}],$
\begin{align}\label{MQK001}
&\frac{|B_{R}\cap\{u(x,t)>M_{1}-2^{-k_{1}}\varepsilon\}|_{\mu_{w_{1}}}}{|B_{R}|_{\mu_{w_{1}}}}\notag\\
&\leq \frac{(1+\overline{C}\varepsilon)(1-\gamma)}{(1-2^{-k_{1}})^{2}}+\overline{C}\bigg(\frac{|A^{a}(M_{1}-\varepsilon,R)|_{\nu_{w_{2}}}}{|Q_{R}|_{\nu_{w_{2}}}}\bigg)^{\frac{1}{3}},
\end{align}
where $\overline{C}=\overline{C}(n,p,\theta_{1},\theta_{2},\lambda,\overline{m},\overline{M}).$ Note that
\begin{align*}
\frac{|A^{a}(M_{1}-\varepsilon,R)|_{\nu_{w_{2}}}}{|Q_{R}|_{\nu_{w_{2}}}}\leq a.
\end{align*}
Take a small positive constant $a$ such that $a^{-1}$ is an integer and
\begin{align*}
\overline{C}a^{\frac{1}{3}}\leq\frac{\gamma}{8}.
\end{align*}
By fixing the value of $a$, we now divide the time interval $[t_{0},t_{0}+R^{\theta_{1}+\theta_{2}}]$ into finite small intervals. Denote $N=a^{-1}$ and $t_{i}=t_{0}+iaR^{\theta_{1}+\theta_{2}}$, $i=1,2,...,N.$

Claim that there exist a small positive constant $\bar{\varepsilon}_{0}$ and a large positive constant $k_{0}>1$ depending only on $n,p,\theta_{1},\theta_{2},\lambda,\gamma,\overline{m},\overline{M}$ such that for any $0<\varepsilon\leq\bar{\varepsilon}_{0}$ and $k_{1}\geq k_{0}$,
\begin{align}\label{DNQ001}
\frac{(1+\overline{C}\varepsilon)(1-\gamma)}{(1-2^{-k_{1}})^{2}}\leq1-\gamma+\frac{\gamma}{8N}.
\end{align}
In fact, since $(1-t)^{-2}\leq(1+6t)$ for $t\in(0,\frac{1}{2})$, then
\begin{align*}
\frac{1+\overline{C}\varepsilon}{(1-2^{-k_{1}})^{2}}\leq(1+\overline{C}\varepsilon)(1+6\cdot2^{-k_{1}}).
\end{align*}
Let $\overline{C}\varepsilon=6\cdot2^{-k_{1}}$. Then we have
\begin{align*}
\frac{1+\overline{C}\varepsilon}{(1-2^{-k_{1}})^{2}}\leq1+2\overline{C}\varepsilon+\overline{C}^{2}\varepsilon^{2}.
\end{align*}
Pick
\begin{align}\label{MWQK001}
\bar{\varepsilon}_{0}=\frac{-2\overline{C}+\sqrt{4\overline{C}^{2}+\frac{\gamma}{2N(1-\gamma)}}}{2\overline{C}^{2}},\quad k_{0}=-\frac{\ln2}{\ln(\overline{C}\bar{\varepsilon}_{0})-\ln6}.
\end{align}
Then we obtain that for any $0<\varepsilon\leq\bar{\varepsilon}_{0}$ and $k_{1}\geq k_{0}$,
\begin{align*}
\frac{1+\overline{C}\varepsilon}{(1-2^{-k_{1}})^{2}}\leq&1+2\overline{C}\bar{\varepsilon}_{0}+\overline{C}^{2}\bar{\varepsilon}^{2}_{0}=1+\frac{\gamma}{8N(1-\gamma)}.
\end{align*}
That is, \eqref{DNQ001} holds.

Consequently, it follows from \eqref{MQK001} that for $0<\varepsilon\leq\bar{\varepsilon}_{0}$, $k_{1}\geq k_{0}$ and $t\in[t_{0},t_{1}]$,
\begin{align*}
&\frac{|B_{R}\cap\{u(x,t)>M_{1}-2^{-k_{1}}\varepsilon\}|_{\mu_{w_{1}}}}{|B_{R}|_{\mu_{w_{1}}}}\leq1-\left(\frac{7}{8}-\frac{1}{8N}\right)\gamma.
\end{align*}
Then applying Lemma \ref{lem005}, we deduce from \eqref{QAZ002} that for any $k_{2}>k_{1}\geq k_{0}$,
\begin{align*}
\frac{|A^{a}(M_{1}-2^{-k_{2}}\varepsilon,R)|_{\nu_{w_{2}}}}{|Q_{R}|_{\nu_{w_{2}}}}\leq& C\left(\frac{|A^{a}(M_{1}-2^{-k_{2}}\varepsilon,R)|_{\nu_{w_{1}}}}{|Q_{R}|_{\nu_{w_{1}}}}\right)^{\frac{\theta_{3}}{\theta_{1}}}\notag\\
\leq& \widehat{C}\left(\frac{\sqrt{a}}{\sqrt[q]{\gamma}(k_{2}-k_{1})^{\frac{2-q}{2q}}}\right)^{\frac{\theta_{3}}{\theta_{1}}},
\end{align*}
where $\widehat{C}=\widehat{C}(n,p,q,\theta_{1},\theta_{2},\theta_{3},\lambda,\overline{m},\overline{M}).$ Pick
\begin{align*}
k_{2}=k_{1}+a^{\frac{q}{2-q}}\gamma^{-\frac{2}{2-q}}\left(\frac{\gamma}{8N\overline{C}\sqrt[3]{\widehat{C}}}\right)^{-\frac{6\theta_{1}q}{\theta_{3}(2-q)}}.
\end{align*}
Then we have
\begin{align}\label{QDWZ001}
\overline{C}\left(\frac{|A^{a}(M_{1}-2^{-k_{2}}\varepsilon,R)|_{\nu_{w_{2}}}}{|Q_{R}|_{\nu_{w_{2}}}}\right)^{\frac{1}{3}}\leq\frac{\gamma}{8N}.
\end{align}
Choose $k_{1}=k_{0}$ and $l_{1}=k_{1}+k_{2}$. By letting $2^{-k_{2}}\varepsilon$ substitute for $\varepsilon$ in \eqref{MQK001}, we have
\begin{align*}
\sup\limits_{t\in[t_{0},t_{1}]}|B_{R}\cap\{u(x,t)>M_{1}-2^{-l_{1}}\varepsilon\}|_{\mu_{w_{1}}}\leq\Big(1-\gamma+\frac{\gamma}{4N}\Big)|B_{R}|_{\mu_{w_{1}}}.
\end{align*}
Then it can be inductively proved that there exist a strictly increasing integer set $\{l_{i}\}_{i=1}^{N}$ such that for $i=1,2,...,N$,
\begin{align*}
\sup\limits_{t\in[t_{i-1},t_{i}]}|B_{R}\cap\{u(x,t)>M_{1}-2^{-l_{i}}\varepsilon\}|_{\mu_{w_{1}}}\leq\Big(1-\gamma+\frac{i\gamma}{4N}\Big)|B_{R}|_{\mu_{w_{1}}}.
\end{align*}
In fact, let the above relation hold in interval $[t_{i-1},t_{i}]$ and then prove that it also holds in the next interval $[t_{i},t_{i+1}]$. For simplicity, denote $\varepsilon_{i}=2^{-l_{i}}\varepsilon$ and $\gamma_{i}=\gamma(1-\frac{i}{4N}).$ Then the assumption implies that
\begin{align*}
|B_{R}\cap\{u(x,t_{i})>M_{1}-\varepsilon_{i}\}|_{\mu_{w_{1}}}\leq(1-\gamma_{i})|B_{R}|_{\mu_{w_{1}}}.
\end{align*}
By the same argument as in \eqref{MQK001}, it follows from \eqref{DNQ001}--\eqref{QDWZ001} that for $\bar{k}_{1}\geq k_{0}$ and $t\in[t_{i},t_{i+1}],$
\begin{align*}
&\frac{|B_{R}\cap\{u(x,t)>M_{1}-2^{-\bar{k}_{1}}\varepsilon_{i}\}|_{\mu_{w_{1}}}}{|B_{R}|_{\mu_{w_{1}}}}\notag\\
&\leq \frac{(1+\overline{C}\varepsilon_{i})(1-\gamma_{i})}{(1-2^{-\bar{k}_{1}})^{2}}+\overline{C}\bigg(\frac{|A^{a}(M_{1}-\varepsilon_{i},R)|_{\nu_{w_{2}}}}{|Q_{R}|_{\nu_{w_{2}}}}\bigg)^{\frac{1}{3}}\notag\\
&\leq1-\gamma_{i}+\frac{\gamma_{i}}{8N}+\overline{C}\bigg(\frac{|A^{a}(M_{1}-2^{-l_{i}}\varepsilon,R)|_{\nu_{w_{2}}}}{|Q_{R}|_{\nu_{w_{2}}}}\bigg)^{\frac{1}{3}}\notag\\
&\leq1-\gamma_{i}+\frac{\gamma_{i}}{8N}+\frac{\gamma}{8N}\notag\\
&<1-\gamma+\frac{i+1}{4N}\gamma.
\end{align*}
where $\overline{C}=\overline{C}(n,p,\theta_{1},\theta_{2},\lambda,\overline{m},\overline{M})$ is defined above and in the third inequality we used the fact that $l_{i}\geq l_{1}>k_{2}$. By taking $\bar{k}_{1}=k_{0}$ and $l_{i+1}=l_{i}+\bar{k}_{1}$, we obtain
\begin{align*}
\sup\limits_{t\in[t_{i},t_{i+1}]}|B_{R}\cap\{u(x,t)>M_{1}-2^{-l_{i+1}}\varepsilon\}|_{\mu_{w_{1}}}\leq\Big(1-\gamma+\frac{(i+1)\gamma}{4N}\Big)|B_{R}|_{\mu_{w_{1}}}.
\end{align*}
Then picking
\begin{align}\label{MWQK002}
\bar{l}_{0}:=&l_{N}=l_{1}+(N-1)k_{0}\notag\\
=&(N+1)k_{0}++a^{\frac{q}{2-q}}\gamma^{-\frac{2}{2-q}}\left(\frac{\gamma}{8N\overline{C}\sqrt[3]{\widehat{C}}}\right)^{-\frac{6\theta_{1}q}{\theta_{3}(2-q)}},
\end{align}
we obtain that \eqref{DQAF001} holds.

\noindent{\bf Step 2.}
For $0<a\leq1$ and $\overline{m}\leq k\leq\overline{M}$, let
\begin{align*}
\tilde{A}^{a}(k,R)=(B_{R}\times[t_{0},t_{0}+a R^{\theta_{1}+\theta_{2}}])\cap\{u<k\}.
\end{align*}
Define $\tilde{v}=(u-(m_{1}+\varepsilon))^{-}$. A direct application of Lemma \ref{lem003} gives that
\begin{align*}
&\sup\limits_{t\in(t_{0},t_{0}+aR^{\theta_{1}+\theta_{2}})}\int_{B_{R}}(\tilde{v}^{2}-C_{0}\tilde{v}^{3})\eta^{2}w_{1}dx\notag\\
&\leq\int_{B_{R}}\tilde{v}^{2}\eta^{2}w_{1}dx\big|_{t_{0}}+C_{0}\int_{B_{R}\times[t_{0},t_{0}+aR^{\theta_{1}+\theta_{2}}]}\tilde{v}^{2}|\nabla\eta|^{2}w_{2}dxdt,
\end{align*}
where $C_{0}=C_{0}(n,p,\lambda,\overline{m},\overline{M})$ and $\eta$ is defined in \eqref{QNE089}. Pick a small constant $0<\bar{\varepsilon}_{1}\leq(2C_{0})^{-1}$, which implies that $1-C_{0}\bar{\varepsilon}_{1}\geq\frac{1}{2}$. Then we obtain that for $t_{0}<t<t_{0}+aR^{\theta_{1}+\theta_{2}}$, $0<\varepsilon\leq\bar{\varepsilon}_{0}$ and $k_{1}>1$,
\begin{align*}
\int_{B_{R}}(\tilde{v}^{2}-C_{0}\tilde{v}^{3})\eta^{2}w_{1}dx\big|_{t}\geq(1-C_{0}\varepsilon)\varepsilon^{2}(1-2^{-k_{1}})^{2}|B_{\sigma R}\cap\{u(x,t)<m_{1}+2^{-k_{1}}\varepsilon\}|_{\mu_{w_{1}}},
\end{align*}
and in view of \eqref{QPZA001},
\begin{align*}
\int_{B_{R}}\tilde{v}^{2}\eta^{2}w_{1}dx\big|_{t_{0}}\leq&\varepsilon^{2}|\{x\in B_{R}: u(x,t_{0})<m_{1}+\varepsilon\}|_{\mu_{w_{1}}}\leq\varepsilon^{2}(1-\gamma)|B_{R}|_{\mu_{w_{1}}},
\end{align*}
and
\begin{align*}
\int_{B_{R}\times[t_{0},t_{0}+aR^{\theta_{1}+\theta_{2}}]}\tilde{v}^{2}|\nabla\eta|^{2}w_{2}dxdt\leq&\frac{C\varepsilon^{2}}{(1-\sigma)^{2}R^{2}}|\tilde{A}^{a}(m_{1}+\varepsilon,R)|_{\nu_{w_{2}}}\notag\\
\leq&\frac{C\varepsilon^{2}R^{\theta_{3}+\theta_{4}-2}}{(1-\sigma)^{2}}|B_{R}|_{\mu_{w_{1}}}\frac{|\tilde{A}^{a}(m_{1}+\varepsilon,R)|_{\nu_{w_{2}}}}{|Q_{R}|_{\nu_{w_{2}}}}\notag\\
=&\frac{C\varepsilon^{2}}{(1-\sigma)^{2}}|B_{R}|_{\mu_{w_{1}}}\frac{|\tilde{A}^{a}(m_{1}+\varepsilon,R)|_{\nu_{w_{2}}}}{|Q_{R}|_{\nu_{w_{2}}}}.
\end{align*}
Therefore, we deduce that for $t_{0}\leq t\leq t_{0}+aR^{\theta_{1}+\theta_{2}}$,
\begin{align*}
&|B_{\sigma R}\cap\{u(x,t)<m_{1}+2^{-k_{1}}\varepsilon\}|_{\mu_{w_{1}}}\notag\\
&\leq|B_{R}|_{\mu_{w_{1}}}\left(\frac{1-\gamma}{(1-C\varepsilon)(1-2^{-k_{1}})^{2}}+\frac{C}{(1-\sigma)^{2}}\frac{|\tilde{A}^{a}(m_{1}+\varepsilon,R)|_{\nu_{w_{2}}}}{|Q_{R}|_{\nu_{w_{2}}}}\right)\notag\\
&\leq|B_{R}|_{\mu_{w_{1}}}\left(\frac{(1+C\varepsilon)(1-\gamma)}{(1-2^{-k_{1}})^{2}}+\frac{C}{(1-\sigma)^{2}}\frac{|\tilde{A}^{a}(m_{1}+\varepsilon,R)|_{\nu_{w_{2}}}}{|Q_{R}|_{\nu_{w_{2}}}}\right),
\end{align*}
and thus,
\begin{align*}
&\frac{|B_{R}\cap\{u(x,t)<m_{1}+2^{-k_{1}}\varepsilon\}|_{\mu_{w_{1}}}}{|B_{R}|_{\mu_{w_{1}}}}\notag\\
&\leq \frac{(1+C\varepsilon)(1-\gamma)}{(1-2^{-k_{1}})^{2}}+\frac{C}{(1-\sigma)^{2}}\bigg((1-\sigma)^{3}+\frac{|\tilde{A}^{a}(m_{1}+\varepsilon,R)|_{\nu_{w_{2}}}}{|Q_{R}|_{\nu_{w_{2}}}}\bigg).
\end{align*}
Take $\sigma$ such that
\begin{align*}
(1-\sigma)^{3}=\frac{|\tilde{A}^{a}(m_{1}+\varepsilon,R)|_{\nu_{w_{2}}}}{|Q_{R}|_{\nu_{w_{2}}}}.
\end{align*}
Then we obtain that for $t_{0}\leq t\leq t_{0}+aR^{\theta_{1}+\theta_{2}},$
\begin{align*}
&\frac{|B_{R}\cap\{u(x,t)<m_{1}+2^{-k_{1}}\varepsilon\}|_{\mu_{w_{1}}}}{|B_{R}|_{\mu_{w_{1}}}}\notag\\
&\leq \frac{(1+\overline{C}\varepsilon)(1-\gamma)}{(1-2^{-k_{1}})^{2}}+\overline{C}\bigg(\frac{|\tilde{A}^{a}(m_{1}+\varepsilon,R)|_{\nu_{w_{2}}}}{|Q_{R}|_{\nu_{w_{2}}}}\bigg)^{\frac{1}{3}},
\end{align*}
where $\overline{C}=\overline{C}(n,p,\theta_{1},\theta_{2},\lambda,\overline{m},\overline{M}).$ Consequently, by the same argument as in the left proof of \eqref{DQAF001} above, we deduce that \eqref{WZPMQ001} holds. The proof is complete.

\end{proof}

A consequence of Lemmas \ref{LEM0035}, \ref{lem005} and \ref{LEM006} gives the improvement on oscillation of $u$ in a small region.
\begin{corollary}\label{AMZW01}
Assume as in Theorem \ref{ZWTHM90} or Theorem \ref{THM060}. Let $0<\gamma<1$, $0<R<\frac{1}{2}$, $-\frac{1}{4}<t_{0}\leq0$ and $\overline{m}\leq m\leq\inf\limits_{B_{2R}\times[t_{0}-R^{\theta_{1}+\theta_{2}},t_{0}]}u\leq\sup\limits_{B_{2R}\times[t_{0}-R^{\theta_{1}+\theta_{2}},t_{0}]}u\leq M\leq\overline{M}$. Then there exist a small constant $\tilde{\varepsilon}_{0}=\tilde{\varepsilon}_{0}(n,p,\theta_{1},\theta_{2},\lambda,\gamma,\overline{m},\overline{M})>0$ and a large constant $l_{0}=l_{0}(n,p,q,\theta_{1},\theta_{2},\theta_{3},\lambda,\gamma,\overline{m},\overline{M})>1$ such that for any $0<\varepsilon\leq\tilde{\varepsilon}_{0}$,

$(i)$ if
\begin{align*}
\frac{|\{x\in B_{R}: u(x,t_{0}-R^{\theta_{1}})>M-\varepsilon\}|_{\mu_{w_{1}}}}{|B_{R}|_{\mu_{w_{1}}}}\leq1-\gamma,
\end{align*}
then
\begin{align*}
\sup\limits_{Q_{R/2}(0,t_{0})}u\leq M-\frac{\varepsilon}{2^{l_{0}}};
\end{align*}

$(ii)$ if
\begin{align*}
\frac{|\{x\in B_{R}: u(x,t_{0}-R^{\theta_{1}})<m+\varepsilon\}|_{\mu_{w_{1}}}}{|B_{R}|_{\mu_{w_{1}}}}\leq1-\gamma,
\end{align*}
then
\begin{align*}
\inf\limits_{Q_{R/2}(0,t_{0})}u\geq m+\frac{\varepsilon}{2^{l_{0}}}.
\end{align*}

\end{corollary}
\begin{proof}
Applying Lemma \ref{LEM0035}, Lemma \ref{lem005} with $a=1$ and Lemma \ref{LEM006}, we obtain that Corollary \ref{AMZW01} holds. In particular, we  fix $q=\frac{3}{2}$ under the assumed conditions in Theorem \ref{THM060}.

\end{proof}

Based on these above facts, we now give the proofs of Theorems \ref{ZWTHM90} and \ref{THM060}, respectively.
\begin{proof}[Proof of Theorem \ref{ZWTHM90}]
Pick a sufficiently large constant $\kappa_{0}\geq2$ such that
\begin{align*}
\frac{\overline{M}-\overline{m}}{\kappa_{0}}<\tilde{\varepsilon}_{0},
\end{align*}
where $\tilde{\varepsilon}_{0}$ is given by Corollary \ref{AMZW01} with $\gamma=\frac{1}{2}$. For $0<R\leq\frac{1}{2}$ and $-\frac{1}{4}<t_{0}<0$, define
\begin{align*}
\overline{\mu}(R)=\sup\limits_{(x,t)\in Q_{R}(0,t_{0})}u(x,t),\quad\underline{\mu}(R)=\inf\limits_{(x,t)\in Q_{R}(0,t_{0})}u(x,t),\quad \omega(R)=\overline{\mu}(R)-\underline{\mu}(R).
\end{align*}
Observe that there is at least one inequality holding in terms of the following two inequalities:
\begin{align}\label{WAMQ001}
|\{x\in B_{R/2}:u(x,t_{0}-(R/2)^{\theta_{1}+\theta_{2}})>\overline{\mu}(R)-\kappa_{0}^{-1}\omega(R)\}|_{\mu_{w_{1}}}\leq\frac{1}{2}|B_{R/2}|_{\mu_{w_{1}}},
\end{align}
and
\begin{align}\label{WAMQ002}
|\{x\in B_{R/2}:u(x,t_{0}-(R/2)^{\theta_{1}+\theta_{2}})<\underline{\mu}(R)+\kappa_{0}^{-1}\omega(R)\}|_{\mu_{w_{1}}}\leq\frac{1}{2}|B_{R/2}|_{\mu_{w_{1}}}.
\end{align}
From Corollary \ref{AMZW01}, it follows that there exists a large constant $l_{0}>1$ such that
\begin{align*}
\overline{\mu}(R/4)\leq\overline{\mu}(R)-\frac{\omega(R)}{\kappa_{0}2^{l_{0}}},\quad\text{if \eqref{WAMQ001} holds,}
\end{align*}
and
\begin{align*}
\underline{\mu}(R/4)\geq\underline{\mu}(R)+\frac{\omega(R)}{\kappa_{0}2^{l_{0}}},\quad\text{if \eqref{WAMQ002} holds.}
\end{align*}
In both cases, we have
\begin{align*}
\omega(R/4)\leq\left(1-\frac{1}{\kappa_{0}2^{l_{0}}}\right)\omega(R)=\frac{1}{4^{\alpha}}\omega(R),\quad\mathrm{with}\;\alpha=-\frac{\ln\big(1-\frac{1}{\kappa_{0}2^{l_{0}}}\big)}{\ln4}.
\end{align*}
Note that for any $0<R\leq\frac{1}{2}$, there is an integer $k$ such that $4^{-(k+1)}\cdot2^{-1}<R\leq4^{-k}\cdot2^{-1}$. In light of the fact that $\omega(R)$ is increasing in $R$, it follows that
\begin{align*}
\omega(R)\leq\omega(4^{-k}\cdot2^{-1})\leq 4^{-k\alpha}\omega(2^{-1})=8^{\alpha}(4^{-(k+1)}\cdot2^{-1})^{\alpha}\omega(2^{-1})\leq CR^{\alpha},
\end{align*}
where $C=C(n,p,q,\theta_{1},\theta_{2},\theta_{3},\lambda,\overline{m},\overline{M})$. Therefore, for any $(x,t)\in B_{1/2}\times(-1/4,t_{0})$, we obtain that

$(i)$ if $|t-t_{0}|\leq2^{-(\theta_{1}+\theta_{2})}$, then
\begin{align*}
|u(x,t)-u(0,t_{0})|\leq&|u(x,t)-u(x,t_{0})|+|u(x,t_{0})-u(0,t_{0})|\leq C(|t-t_{0}|^{\frac{\alpha}{\theta_{1}+\theta_{2}}}+|x|^{\alpha})\notag\\
\leq& C\big(|x|+|t-t_{0}|^{\frac{1}{\theta_{1}+\theta_{2}}}\big)^{\alpha};
\end{align*}

$(ii)$ if $|t-t_{0}|>2^{-(\theta_{1}+\theta_{2})}$, there exists a set $\{t_{i}\}_{i=1}^{N}$ such that $t<t_{1}\leq\cdots\leq t_{N}<t_{0}$,
\begin{align*}
|u(x,t)-u(0,t_{0})|\leq&|u(x,t)-u(x,t_{1})|+|u(x,t_{1})-u(x,t_{0})|+|u(x,t_{0})-u(0,t_{0})|\notag\\
\leq&C\big(|t-t_{1}|^{\frac{\alpha}{\theta_{1}+\theta_{2}}}+|t_{1}-t_{0}|^{\frac{\alpha}{\theta_{1}+\theta_{2}}}+|x|^{\frac{\alpha}{\theta_{1}+\theta_{2}}}\big)\notag\\
\leq&C\big(|x|+|t-t_{0}|^{\frac{1}{\theta_{1}+\theta_{2}}}\big)^{\alpha},\quad\text{if}\;N=1,
\end{align*}
and
\begin{align*}
&|u(x,t)-u(0,t_{0})|\notag\\
&\leq|u(x,t)-u(x,t_{1})|+\sum^{N-1}_{i=1}|u(x,t_{i})-u(x,t_{i+1})|\notag\\
&\quad+|u(x,t_{N})-u(x,t_{0})|+|u(x,t_{0})-u(0,t_{0})|\notag\\
&\leq C\Big(|t-t_{1}|^{\frac{\alpha}{\theta_{1}+\theta_{2}}}+\sum^{N-1}_{i=1}|t_{i}-t_{i+1}|^{\frac{\alpha}{\theta_{1}+\theta_{2}}}+|t_{N}-t_{0}|^{\frac{\alpha}{\theta_{1}+\theta_{2}}}+|x|^{\frac{\alpha}{\theta_{1}+\theta_{2}}}\Big)\notag\\
&\leq C\big(|x|+|t-t_{0}|^{\frac{1}{\theta_{1}+\theta_{2}}}\big)^{\alpha},\quad \text{if}\;N\geq2.
\end{align*}
The proof is complete.

\end{proof}

\begin{proof}[Proof of Theorem \ref{THM060}]
To begin with, applying the aforementioned proof of Theorem \ref{ZWTHM90} with minor modification, we also obtain that there exists a small constant $0<\alpha<1$ and a large constant $C>0$, both depending only on $n,p,\theta_{2},\lambda,\overline{m},\overline{M},$ such that for any $t_{0}\in (-1/4,0)$,
\begin{align}\label{QM916}
|u(x,t)-u(0,t_{0})|\leq C\left(|x|+\sqrt[\theta_{2}]{|t-t_{0}|}\right)^{\alpha},\quad\forall\, (x,t)\in B_{1/2}\times(-1/4,t_{0}].
\end{align}
For $R\in(0,1/2)$, $(y,s)\in Q_{1/R}$, define
\begin{align*}
u_{R}(y,s)=u(Ry,R^{\theta_{2}}s),\quad A_{R}(y)=A(Ry).
\end{align*}
Therefore, $u_{R}$ verifies
\begin{align*}
|y|^{\theta_{2}}\partial_{s}u^{q}_{R}-\mathrm{div}(A_{R}|y|^{2}\nabla u_{R})=0,\quad\mathrm{in}\;Q_{1/R}.
\end{align*}
By the change of variables, we obtain that this equation keeps uniformly parabolic in $B_{1/2}(\bar{y})\times(-R^{-\theta_{2}},0)$ for any $\bar{y}\in\partial B_{1}$.

For any $(x,t),(\tilde{x},\tilde{t})\in B_{1/2}\times(-1/4,0),$ let $|\tilde{x}|\leq|x|$ without loss of generality. Write $R=|x|$. It then follows from the interior H\"{o}lder estimates for uniformly parabolic equations that there exist two constants $0<\beta=\beta(n,p,\theta_{2},\lambda,\overline{m},\overline{M})<1$ and $0<C=C(n,p,\theta_{2},\lambda,\overline{m},\overline{M})$ such that for any $\bar{y}\in\partial B_{1}$ and $\bar{s}\in(-4^{-1} R^{-\theta_{2}},0)$,
\begin{align}\label{WAQA001}
|u_{R}(y,s)-u_{R}(\bar{y},\bar{s})|\leq C(|y-\bar{y}|+\sqrt{|s-\bar{s}|})^{\beta},
\end{align}
for any $(y,s)$ satisfying that $|y-\bar{y}|+\sqrt{|s-\bar{s}|}<1/2$.

Observe that for any $(x,t),(\tilde{x},\tilde{t})\in B_{1/2}\times(-1/4,0),$
\begin{align*}
|u(x,t)-u(\tilde{x},\tilde{t})|\leq|u(x,t)-u(x,\tilde{t})|+|u(x,\tilde{t})-u(\tilde{x},\tilde{t})|.
\end{align*}
On one hand, if $|t-\tilde{t}|\leq R^{2\theta_{2}}$, then we deduce from \eqref{WAQA001} that
\begin{align*}
|u(x,t)-u(x,\tilde{t})|\leq&\left|u_{R}(x/R,t/R^{\theta_{2}})-u_{R}(x/R,\tilde{t}/R^{\theta_{2}})\right|\notag\\
\leq&C|(t-\tilde{t})/R^{\theta_{2}}|^{\beta/2}\leq C|t-\tilde{t}|^{\beta/4},
\end{align*}
while, if $|t-\tilde{t}|>R^{2\theta_{2}}$, then we have from \eqref{QM916} that
\begin{align*}
&|u(x,t)-u(x,\tilde{t})|\notag\\
&\leq|u(x,t)-u(0,t)|+|u(0,t)-u(0,\tilde{t})|+|u(0,\tilde{t})-u(x,\tilde{t})|\notag\\
&\leq C\big(R^{\alpha}+|t-\tilde{t}|^{\frac{\alpha}{\theta_{2}}}\big)\leq C|t-\tilde{t}|^{\frac{\alpha}{2\theta_{2}}}.
\end{align*}

On the other hand, if $|x-\tilde{x}|\leq R^{2}$, then it follows from \eqref{WAQA001} that
\begin{align*}
|u(x,\tilde{t})-u(\tilde{x},\tilde{t})|=&\left|u_{R}(x/R,\tilde{t}/R^{\theta_{2}})-u_{R}(\tilde{x}/R,\tilde{t}/R^{\theta_{2}}\big)\right|\notag\\
\leq&C|(x-\tilde{x})/R|^{\beta}\leq C|x-\tilde{x}|^{\beta/2},
\end{align*}
while, if $|x-\tilde{x}|>R^{2}$, then we see from \eqref{QM916} that
\begin{align*}
|u(x,\tilde{t})-u(\tilde{x},\tilde{t})|\leq&|u(x,\tilde{t})-u(0,\tilde{t})|+|u(0,\tilde{t})-u(\tilde{x},\tilde{t})|\notag\\
\leq& C\big(R^{\alpha}+|\tilde{x}|^{\alpha}\big)\leq CR^{\alpha}\leq C|x-\tilde{x}|^{\frac{\alpha}{2}}.
\end{align*}
Consequently, we complete the proof Theorem \ref{THM060}.

\end{proof}

\noindent{\bf{\large Acknowledgements.}} C. Miao was supported by the National Key Research and Development Program of China (No. 2020YFA0712900) and NSFC Grant 11831004. Z. Zhao was partially supported by CPSF (2021M700358).



\end{document}